\documentclass[12pt]{amsart}
\usepackage{color,graphicx,array, amssymb, amscd,slashed,amsmath,amscd,latexsym,multirow,mathrsfs}
\usepackage[hidelinks]{hyperref}
\usepackage{tikz}
\usepackage{tikz-cd}
\usetikzlibrary{arrows.meta,decorations.markings}
\usepackage[hang,small]{caption}
\usepackage[subrefformat=parens]{subcaption}
\usepackage{placeins}
\usepackage{needspace}
\usetikzlibrary{fit,matrix}
\usepackage{enumitem}
\usepackage{bm}
\newlist{steps}{enumerate}{1}
\setlist[steps, 1]{label = Step \arabic*:}

\newtheorem{theorem}{Theorem}[section]
\newtheorem{corollary}{Corollary}[section]
\newtheorem{lemma}{Lemma}[section]
\newtheorem{proposition}{Proposition}[section]
\newtheorem{remark}{Remark}[section]
\newtheorem{example}{Example}[section]
\newtheorem{definition}{Definition}[section]

\numberwithin{equation}{section}

\numberwithin{equation}{section}
\newcommand{\R}{\mathbb R}

\newcommand{\C}{\mathbb C}

\newcommand{\Z}{\mathbb Z}

\newcommand{\tr}{\operatorname{tr}}

\renewcommand{\Re}{\operatorname{Re}}
\renewcommand{\Im}{\operatorname{Im}}

\newcommand{\SU}{\mathrm {SU}(2)}

\numberwithin{equation}{section}

\newcommand{\Ab}{\boldsymbol{\mathcal{A}}}

\newcommand{\ii}{\mathrm{i}}

\newcommand{\dd}{\mathrm{d}}

\newcommand{\LamR}{\Lambda_0}

\usepackage{amsmath}	
\DeclareRobustCommand{\rev}[1]{{#1}}

\graphicspath{{./}}
\begin{document}
\tikzset{->-/.style 2 args={
    postaction={decorate},
    decoration={markings, mark=at position #1 with {\arrow{>}}}
    },
    ->-/.default={0.5}{}
}

\tikzset{-<-/.style 2 args={
    postaction={decorate},
    decoration={markings, mark=at position #1 with {\arrow{<}}}
    },
    -<-/.default={0.5}{}
}
\title{Finite-Gap Solutions of the Pohlmeyer--Lund--Regge Equation and the Associated \rev{Lund--Regge Surfaces}}
\dedicatory{\rev{Dedicated to Professor Jun-ichi Inoguchi\\
on the occasion of his sixtieth birthday.}}
\textbf{}
\author[Y.~Kogo]{Yuhei Kogo}
\address{Department of Mathematics, Hokkaido University, Sapporo, 060-0810, Japan}
\email{yh1123\_kg\_5813@eis.hokudai.ac.jp}
\thanks{The author is supported by JST SPRING, Grant Number JPMJSP2119.}

 \date{\today}
\pagestyle{plain}

\begin{abstract}
We develop a finite-gap construction for the
Pohlmeyer--Lund--Regge (PLR) equation and the associated \rev{Lund--Regge surfaces}.
From the hyperelliptic spectral data we build a Baker--Akhiezer function and an \(\SU\)-frame,
yielding an explicit theta-quotient formula for the PLR solution.
We then derive \rev{criteria for \(s\)-\rev{periodicity} and \(t\)-periodicity}:
\rev{under the stated theta-period assumptions,} \(s\)-\rev{periodicity} and \(t\)-periodicity are each
equivalent to a critical-point condition for the corresponding \rev{normalized differential}
together with a \rev{phase condition} at the reconstruction point.
This provides a PLR analogue of the \rev{periodicity} mechanism of Calini--Ivey.
\end{abstract}

\maketitle

\section{Introduction}
Many integrable nonlinear PDEs admit geometric realizations in terms of moving curves and surfaces \cite{Lamb,NSW}.
A classical example is the \emph{Hasimoto correspondence} \cite{Hashimoto}:
the vortex filament equation (VFE), i.e.\rev{,}\ binormal motion of an arclength
parametrized space curve, is transformed by the Hasimoto map into the focusing
nonlinear Schr\"odinger equation (NLS).
This correspondence makes the algebro--geometric
(finite-gap) method available in the geometric setting and, in particular,
allows one to reconstruct space curves from spectral data and to study
geometric properties such as \rev{periodicity}. For finite-gap VFE
filaments, Calini--Ivey \cite{CaliniIvey} gave a particularly clear \rev{periodicity}
theory in terms of distinguished spectral points and a quantization condition.

\medskip
In this paper we study the analogous finite-gap problem for the
\emph{Pohlmeyer--Lund--Regge (PLR)} equation.
\begingroup

In the complex-potential formulation used below, it is
\[
q_{st}+\tfrac12 q\int^s (|q|^2)_t\,\dd u=0.
\]
The real integration constant is specified in Section~\ref{subsec:LR}.
\endgroup
The PLR equation appears, for instance, in the Pohlmeyer reduction of
relativistic strings \cite{Pohlmeyer} and in the Lund--Regge model
\cite{LundRegge,Lund}.
It has also been studied from a geometric point of view: Fukumoto and
Miyajima \cite{FukumotoMiyajima} related the Lund--Regge equation to the
localized induction hierarchy, while Chen and Li \cite{Chen-Lin} investigated
the associated Lund--Regge surface and its evolution.
\begingroup

In our previous work \cite{KKM25} we studied families of curves whose swept-out
surfaces satisfy the \emph{Lund--Regge condition}
\[
\gamma_{st}=\gamma_s\times\gamma_t,
\]
and
\endgroup
introduced a Hasimoto-type complex potential \(q(s,t)\), derived the
\rev{PLR equation} satisfied by \(q\), and exhibited its \(2\times2\) Lax pair.
Moreover, we proved a Sym-type reconstruction formula \cite{Sym} recovering the curve from
a solution of the Lax system.
\rev{These results provide a geometric Lax representation in which to
revisit PLR finite-gap solutions and study the \rev{periodicity} of the associated
coordinate curves.}

\medskip
\begingroup
Date already constructed quasi-periodic solutions of the PLR equation
in \cite[Section~5, Theorem~5.1]{Date}. Our finite-gap formulas build on
that Baker--Akhiezer (BA) scheme for sine-Gordon type systems and on the
general algebro-geometric method for $2\times2$ Lax pairs
\cite{Belokolos}. Here we express the construction in the complex-potential
Lax representation of \cite{KKM25}, fixing the common integration paths
and the reality and normalization conditions needed for an
$\SU$ frame and the scaled Sym reconstruction.
This gives explicit theta-function formulas for the potential $q$ and
the reconstructed map $\gamma(s,t)$.
The geometric results below concern the periodicity of this map and its
coordinate curves.
\endgroup

\medskip
\rev{The main geometric results give a spectral theory of \rev{periodicity} for} \rev{coordinate curves of these surfaces. Here an \(s\)-curve means \(s\mapsto\gamma(s,t)\) at fixed \(t\), and a \(t\)-curve means \(t\mapsto\gamma(s,t)\) at fixed \(s\).}
\rev{In the VFE setting, one usually assumes a periodic potential and studies
whether the reconstructed position is also periodic in the arclength
variable; see}
\cite{CaliniIvey}.
\rev{The geometric condition is invariant under
\(\widehat\gamma(s,t)=-\gamma(-t,s)\). This symmetry motivates examining
periodicity in both parameters, although it need not preserve the
arclength normalization in \(s\). The potential formulation uses the
Frenet frame of the \(s\)-curves and is not manifestly symmetric.}
We derive both an \(s\)-\rev{periodicity} criterion for the space curve
\(\gamma(\cdot,t)\) and a criterion for \(t\)-\rev{periodicity of the coordinate curve}
\(\gamma(s,\cdot)\) at a fixed reconstruction point.
\rev{In either parameter, periodicity of the reconstructed position means:
\(\gamma(s+L,t)=\gamma(s,t)\) or \(\gamma(s,t+T)=\gamma(s,t)\), respectively.}
These criteria are expressed in terms of two Abelian integrals
\(\Omega_1,\Omega_2\) evaluated at the reconstruction point:
\rev{the criteria combine vanishing of the corresponding differential
$d\Omega_1$ or $d\Omega_2$ there with the phase condition
$L\Omega_1\in2\pi\Z$ or $T\Omega_2\in2\pi\Z$, respectively.}
\rev{Under the theta-period assumptions in Theorems~\ref{thm:closure-s}
and~\ref{thm:closure-t}, these results extend the Calini--Ivey
\rev{periodicity} mechanism to the PLR reconstruction. The $t$-criterion uses the
second normalized differential $d\Omega_2$, with its pole at zero,
in addition to the $s$-criterion associated with $d\Omega_1$ at infinity.
Corollary~\ref{cor:coordinate-closure} also treats a complete regular
coordinate curve that passes through singular points of the surface map.}

\medskip
A natural question in this framework is whether there exist \rev{Lund--Regge surfaces} that are
doubly periodic in \((s,t)\).
In contrast to \cite{CaliniIvey}, which treats spatial \rev{periodicity} for VFE, the PLR
setting leads to the problem of imposing both \(s\)-\rev{periodicity} and
\(t\)-periodicity simultaneously.
We derive explicit necessary and sufficient conditions for each of these
properties, but we do not exhibit a spectral datum or reconstruction point for
which both hold at the same time.
We nevertheless treat the genus-one case explicitly, expressing all relevant
quantities in terms of elliptic integrals and \rev{theta functions}, and
present numerical examples of \(s\)-\rev{periodicity} and \(t\)-periodicity taken
separately.


\subsection*{Organization of the paper}
Section~\ref{sec:PLR-LR} recalls the PLR equation, the \rev{Lund--Regge condition}, the
PLR potential, the Lax pair, and the Sym reconstruction formula.
Section~\ref{sec:finite-gap} develops the finite-gap construction via a
hyperelliptic spectral curve, normalized differentials, and a BA function, and
derives theta-function formulas for \(q\) and the associated frame.
Section~\ref{sec:closure} establishes the spectral criteria for \(s\)-\rev{periodicity}
and \(t\)-periodicity at a reconstruction point.
Section~\ref{sec:genus1} specializes to genus one and provides explicit elliptic
formulas.
Finally, Section~\ref{sec:examples} presents numerical examples illustrating the
\rev{periodicity} phenomena.
\rev{Appendix~\ref{app:reflected-cycles} gives the reflected-cycle calculation on the punctured spectral curve.}

\section{PLR equation and \rev{Lund--Regge condition}}\label{sec:PLR-LR}

In this section we recall the geometric formulation of the Pohlmeyer--Lund--Regge
equation and the associated Lax pair from our previous work~\cite{KKM25}.
Throughout the paper we identify $\R^3$ with $\mathfrak{su}(2)$ by
\begin{equation}\label{eq:R3-su2}
  (p,q,r)^{\mathsf T}\in\R^3
  \ \longleftrightarrow\
\rev{\frac12\begin{pmatrix}
    \ii p & \ii q-r\\[0.2em]
    \ii q+r & -\ii p
  \end{pmatrix}}\in\mathfrak{su}(2),
\end{equation}
so that the cross product and the Euclidean inner product are represented as
\[
  a\times b = [a,b],\qquad
  \langle a,b\rangle = -2\tr(ab),
  \qquad a,b\in\mathfrak{su}(2).
\]
\begingroup

Write \(j\) for the identification in \eqref{eq:R3-su2}, and define
\(A_v\) by \(A_v w=v\times w\). The identification
$i:\mathfrak{so}(3)\to\mathfrak{su}(2)$ in
\cite[Eq.~(2.14)]{KKM25} then satisfies \(i(A_v)=j(v)\).
Thus \(j(A_vw)=[i(A_v),j(w)]\).
The metric here is \(-2\operatorname{tr}(ab)\); for the conventional Killing
form \(B(a,b)=4\operatorname{tr}(ab)\), it is \(-B/2\).
\endgroup

\subsection{\rev{Lund--Regge condition} and Hasimoto-type variable}\label{subsec:LR}

Let
\[
  \gamma:\R^2\to\R^3,\qquad (s,t)\longmapsto\gamma(s,t),
\]
be a smooth family of curves parametrized by arclength.
We consider the \emph{\rev{Lund--Regge condition}}
\begin{equation}\label{eq:LR}
  \gamma_{st} = \gamma_s\times\gamma_t,
\end{equation}
introduced in~\cite{LundRegge} as a geometric realization of the
Pohlmeyer--Lund--Regge equation and studied in detail in~\cite{KKM25}.
Let $\kappa$ and $\tau$ denote the curvature and torsion of $\gamma$, and let
$(T,N,B)$ be the Frenet frame along $\gamma$.  In~\cite{KKM25} we proved
that \rev{the condition}~\eqref{eq:LR} is equivalent to a certain system of
nonlinear \rev{integro-differential equations \cite[Eqs.~(3.8)--(3.9)]{KKM25}} for $(\kappa,\tau)$, and that
\rev{this system can be encoded by a Hasimoto-type complex potential.}
\begingroup

This condition does not specify \(\gamma_t\) from a single initial curve.
Indeed, writing \(\gamma_t=aT+bN+cB\) gives
\[
a_s=b\kappa,\qquad b_s=c(\tau-1)-a\kappa,\qquad c_s=b(1-\tau).
\]
For each fixed $t$, this linear system in $s$ also requires the values
$a(s_*,t),b(s_*,t),c(s_*,t)$ at a reference point $s_*$.
\endgroup

\begin{definition}[PLR potential]\label{def:q}
Let $\gamma$ be a solution of the \rev{Lund--Regge condition}~\eqref{eq:LR} with
curvature $\kappa$ and torsion $\tau$.
\rev{On a simply connected domain where $\kappa>0$, choose a real phase $p$ with
$p_s=\tau-1$ and $p_t=-c/\kappa$. Such a phase exists because
$\tau_t=-(c/\kappa)_s$ \cite[Eq.~(3.4)]{KKM25}, and is unique up to
an additive constant. This choice fixes the arbitrary function of $t$
in the indefinite integral discussed in \cite[Remark~3.2(1)]{KKM25}.
Indeed, for any local primitive $p$ satisfying $p_s=\tau-1$,
the compatibility identity above gives
$\partial_s(p_t+c/\kappa)=0$, hence $p_t=-c/\kappa+f(t)$.
Subtracting an antiderivative of $f(t)$ from $p$ preserves $p_s$ and
gives $p_t=-c/\kappa$, leaving only an additive real constant.}
We define the \emph{PLR potential}
\begin{equation}\label{eq:q-def}
  q(s,t)\rev{=\kappa(s,t)e^{\ii p(s,t)}}
   = \kappa(s,t)\,
     \exp\left(
       \ii\int^s (\tau(u,t)-1)\,\dd u
     \right).
\end{equation}
\rev{The integral in this formula denotes the phase $p$ just specified.}
\end{definition}

The quantity $q$ plays the same role as the Hasimoto variable for the vortex
filament equation.  In particular, $|q|=\kappa$ and
\begin{equation}\label{eq:kappa-tau-from-q}
  \kappa = |q|,\qquad
  \tau = 1 + \partial_s\arg q.
\end{equation}

In terms of $q$ the \rev{Lund--Regge condition} reduces to a single complex
equation.
\rev{The indefinite integral below has a real integration constant,
specified after the proposition.}

\begin{proposition}[Complex PLR equation {\cite[Cor.~3.1]{KKM25}}]
\label{prop:complex-PLR}
Let $\gamma$ \rev{satisfy the Lund--Regge condition} and let $q$ be the PLR potential
defined by~\eqref{eq:q-def}. Then $q$ satisfies
\begin{equation}\label{eq:PLR-q}
  q_{st} + \frac12\,q\int^s (|q|^2)_t\,\dd u = 0.
\end{equation}
\rev{Conversely, on a simply connected domain where $q\ne0$, a solution
of~\eqref{eq:PLR-q} with the prescribed real integration constant
admits a local arclength-parametrized reconstruction satisfying
\eqref{eq:LR}. The reconstruction is understood up to a rigid motion
in $\R^3$. Its curvature and torsion are given by
\eqref{eq:kappa-tau-from-q}.}
\end{proposition}

\begingroup

The local formulation of \eqref{eq:PLR-q} is
\[
q_{st}=Rq,\qquad R_s=-\tfrac12(|q|^2)_t,\qquad R\in\R.
\]
For the geometric potential above, $\kappa_t=-b$ gives
$q_t=(-b-\ii c)e^{\ii p}$ and $q_{st}=aq$, so $R=a$.
On a local parameter rectangle, the integral in \eqref{eq:PLR-q} means
\[
\int_{s_*}^s(|q(u,t)|^2)_t\,du-2R(s_*,t)=-2R(s,t).
\]
Its integration constant is therefore part of the formulation.
\endgroup

\begin{remark}
\begingroup

The condition $\tau\equiv1$ implies $p_s=0$, but does not imply
that $q=\kappa e^{\ii p}$ is real-valued, since $p_t=-c/\kappa$.
The real-valued reduction requires a constant phase, which can then be
normalized to zero. For the sine-Gordon reduction under this additional
restriction, see~\cite[Remark~3.2]{KKM25}.
\endgroup
\end{remark}

\subsection{Lax pair in terms of the PLR potential}\label{subsec:Lax}

The complex equation~\eqref{eq:PLR-q} is integrable.
\begingroup

Following~\cite{KKM25}, we introduce a frame
\[
F(s,t;\lambda)\in\mathrm{SL}(2,\C),\qquad\lambda\in\C^\times,
\]
holomorphic in \(\lambda\) locally and smooth in \((s,t)\), with
\(F(s,t;\lambda)\in\SU\) for real \(\lambda\ne0\),
\endgroup
satisfying the Lax pair
\begin{equation}\label{eq:Lax-PLR}
  (F)_s = F L,\qquad
  (F)_t = F M,
\end{equation}
with
\begin{equation}\label{eq:L-M-PLR}
  L
   = \frac12\begin{pmatrix}
       \ii\lambda & q\\[0.2em]
       -\bar q & -\ii\lambda
     \end{pmatrix},\qquad
  M
   = \frac{\ii}{2\lambda}
     \begin{pmatrix}
       -\Re\bigl(q_{st}/q\bigr) & -q_t\\[0.2em]
       -\overline{q_t} & \Re\bigl(q_{st}/q\bigr)
     \end{pmatrix}.
\end{equation}

\begin{proposition}[{\cite[Cor.~3.2]{KKM25}}]
\label{prop:Lax-representation}
The compatibility condition
\[
  (F_{s})_{t} = (F_{t})_{s}
  \quad\Longleftrightarrow\quad
  L_{t} - M_{s} + \rev{[M,L]} = 0
\]
\begingroup

is equivalent, where \(q\ne0\), to
\[
q_{st}=Rq,\qquad R_s+\tfrac12(|q|^2)_t=0,
\qquad R=\Re(q_{st}/q).
\]
In particular $q_{st}/q$ is real. The third-order equation
$(q_{st}/q)_s+\tfrac12(|q|^2)_t=0$ alone gives only
$\partial_s\Im(q_{st}/q)=0$; it allows an imaginary part depending on $t$. The local system above is equivalent to \eqref{eq:PLR-q} when the real
integration constant is chosen as in Section~\ref{subsec:LR}.
Conversely, on a simply connected parameter domain, a solution of these
equations determines \(F\) after its value at one point is prescribed.
For each real \(\lambda\ne0\), unitary initial data give an \(\SU\)-valued
frame. Two such families differ by a left factor depending on \(\lambda\)
but independent of \((s,t)\).
\endgroup
\end{proposition}

In our previous paper~\cite{KKM25} we also showed that, after an
appropriate diagonal gauge transformation, the frame $F$ coincides
with the gauged Frenet frame of \rev{an \(s\)-curve on a Lund--Regge surface} at $\lambda=1$.
Therefore the spectral problem associated with~\eqref{eq:Lax-PLR} encodes both
the PLR equation for the potential $q$ and the geometry of the curve
$\gamma$.

\subsection{Sym representation of \rev{Lund--Regge surfaces}}\label{subsec:Sym}

The geometric content of the Lax pair~\eqref{eq:Lax-PLR} is summarized by a
\rev{Sym-type} representation formula.

\begin{theorem}[Sym formula {\cite[Thm.~3.2]{KKM25}}]
\label{thm:Sym}
Let $q$ be a solution of the PLR equation~\eqref{eq:PLR-q} and let
$F$ be a solution of the Lax pair~\eqref{eq:Lax-PLR}\rev{, with \(F\in\SU\) for real \(\lambda\ne0\) and differentiable dependence on \(\lambda\)}.  Define
\begin{equation}\label{eq:Sym-PLR}
  \gamma(s,t)
   = \left.
      \bigl(\partial_{\lambda} F(s,t)\bigr)
      F(s,t)^{-1}\right|_{\lambda=1}
   \in\mathfrak{su}(2),
\end{equation}
where $\partial_{\lambda} = \partial/\partial\lambda$.
Under the identification~\eqref{eq:R3-su2}, the map
\rev{$\gamma$ is defined on the parameter domain of $F$ and is}
an arclength-parametrized solution of the
\rev{Lund--Regge condition}~\eqref{eq:LR}. Its curvature and torsion
\rev{where $q\ne0$} are related to $q$ by~\eqref{eq:kappa-tau-from-q}.
\rev{We restrict the surface to its immersion locus $q_t\ne0$.
Conversely, locally on a simply connected domain with positive
curvature, an arclength-parametrized family satisfying the Lund--Regge
condition arises in this way from a solution of~\eqref{eq:PLR-q}.}
\end{theorem}

\begingroup

\begin{remark}[Reconstruction point]\label{rem:reconstruction-point}
For a fixed real \(\Lambda_0>0\), let
\[
X=\left.(\partial_\lambda F)F^{-1}\right|_{\lambda=\Lambda_0},
\qquad \gamma_{\Lambda_0}=\Lambda_0 X.
\]
Writing \(A=L_\lambda=\operatorname{diag}(\ii,-\ii)/2\), direct differentiation gives
\[
X_s=FAF^{-1},\quad X_t=-\Lambda_0^{-1}FMF^{-1},\quad
X_{st}=F[M,A]F^{-1}=\Lambda_0[X_s,X_t].
\]
Consequently \((\gamma_{\Lambda_0})_{st}=
[ (\gamma_{\Lambda_0})_s,(\gamma_{\Lambda_0})_t]\).
The unscaled map \(X\) satisfies the Lund--Regge condition at
\(\Lambda_0=1\), or in the degenerate case \([X_s,X_t]=0\).
The scaled map has \(|(\gamma_{\Lambda_0})_s|=\Lambda_0\);
its arclength coordinate is \(S=\Lambda_0s\), with
\[
\kappa=|q|/\Lambda_0,\qquad
\tau=1+\Lambda_0^{-1}\partial_s\arg q.
\]
We use this scaled reconstruction throughout the finite-gap part.
Scaling by a nonzero constant does not change periodicity in \(s\) or \(t\).
\end{remark}
\endgroup
\begin{remark}
In particular, the family $\{\gamma(\,\cdot\,,t)\}_{t\in\R}$ is \rev{the family of \(s\)-curves associated with} the finite-gap solution $q$ that we construct in
Section~\ref{sec:finite-gap}.  The \rev{periodicity} and geometric properties of
$\gamma$ will be studied in Sections~\ref{sec:closure} and~\ref{sec:genus1}.
\end{remark}

\section{Finite-gap solutions}\label{sec:finite-gap}

In this section we construct a class of quasi-periodic (finite-gap) solutions
of the PLR equation~\eqref{eq:PLR-q} and the associated \rev{Lund--Regge surfaces}.
Our construction follows the standard algebro-geometric scheme for
\(2\times 2\) Lax pairs (see, for example, \cite{Belokolos,Date}),
adapted to the reality conditions of the PLR Lax pair~\eqref{eq:Lax-PLR}.
\rev{The construction proceeds from the marked spectral curve and normalized
differentials to the common paths and divisor data, then to the BA
function. The scalar vanishing lemma gives uniqueness and the Lax
equations; real normalization then gives the $\SU$ frame and the
reconstructed map. The long topological verification of the cycle
conventions is separated into Appendix~\ref{app:reflected-cycles}.}

\subsection{Spectral curve and real structure}\label{subsec:spectral-curve}
Let \(\mathcal{R}\) be a genus-\(g\) hyperelliptic Riemann surface defined by
\[
\mu^2 = \prod_{j=1}^{g+1} (\lambda - \lambda_j)(\lambda - \bar\lambda_j), \quad 
\lambda_j \neq \lambda_k(j\neq k),\ \lambda_j\neq\bar\lambda_k,\lambda_j \neq 0.
\]
\rev{From each conjugate pair choose the representative with
$\Im\lambda_j>0$. This relabeling leaves the spectral curve unchanged;
the cuts and their orientations below use these upper-half-plane
representatives.}

Let \(P_{\infty}^\pm\) and \(P_0^\pm\) be the points over \(\lambda = \infty\) and \(0\), with local parameters \(z = \lambda^{-1}\) near \(P_{\infty}^\pm\), and \(\lambda\) near \(P_0^\pm\).
\begingroup

The signs are fixed by
\(\mu/\lambda^{g+1}\to\pm1\) at \(P_\infty^\pm\), and
\(P_0^\pm=(0,\pm\mu_0)\), where
\(\mu_0=\prod_{j=1}^{g+1}|\lambda_j|>0\).
Let \(\iota(\lambda,\mu)=(\lambda,-\mu)\).
\endgroup
There exists a fixed-point-free anti-holomorphic involution
\[
\rev{\sigma : (\lambda, \mu) \mapsto (\bar\lambda, -\bar\mu)}
\]
on the Riemann surface \(\mathcal{R}\).

Let \(\{a_1, \dots, a_g, b_1, \dots, b_g\}\) be a homology basis of \(\mathcal{R}\), satisfying
\[
a_i \cdot a_j = b_i \cdot b_j = 0, \quad a_i \cdot b_j = \delta_{ij}.
\]

We choose a canonical basis such that
\begin{align*}
\sigma_*(a_j) &\rev{\equiv a_j}, \\
\sigma_*(b_j) &\rev{\equiv -b_j + \sum_{k \ne j} a_k}, \quad j = 1, \dots, g.
\end{align*}

\rev{This homology basis is illustrated in Figure~\ref{fig:homologybasis}.}
\begingroup

Here $H_1(\mathcal R,\Z)$ is the first homology group with integer
coefficients, and $\equiv$ denotes equality of homology classes.
To explain the action, write
\[
\rho(\lambda,\mu)=(\bar\lambda,\bar\mu),\qquad \sigma=\iota\rho.
\]
Let $S_+$ be the sheet with $\mu/\lambda^{g+1}\to1$ at infinity and
$S_-=\iota S_+$. Thus $\rho$ reflects the projection and preserves
each sheet, whereas $\iota$ exchanges the sheets.
In the convention of
Figure~\ref{fig:homologybasis}, solid cycle arcs lie on $S_-$ and
dashed cycle arcs on $S_+$. We use this convention also in
Figures~\ref{fig:cycle-reflection} and~\ref{fig:homologybasisg1}.
Reflection by $\rho$ therefore
preserves the line style; it does not exchange the sheets.
The vertical segments are branch cuts, and the infinity label in
Figure~\ref{fig:homologybasis} is the projection of $P_\infty^\pm$.
The cuts join conjugate branch points, with the omitted cut
$(\lambda_{g+1},\bar\lambda_{g+1})$ on the left.
Each $a_j$ is drawn clockwise on $S_-$ in
Figures~\ref{fig:homologybasis} and~\ref{fig:cycle-reflection}(a).
Equivalently, it is represented by a counterclockwise loop about
the same cut on $S_+$.
The $b_j$ path leaves the right bank of the omitted cut on $S_+$,
passes below intervening cuts to the left bank of cut $j$, and returns
on $S_-$ from its right bank to the left bank of the omitted cut,
above the cuts and outside the two endpoint cuts.
Representatives are nested and slightly separated, with
$a_k\cdot b_j=\delta_{kj}$.
For general branch positions these disjoint, conjugation-invariant
cuts and paths are transported from this configuration by an
orientation-preserving homeomorphism of the upper half-plane, extended
by conjugation; this changes only the paths, not the complex coordinate.

\rev{For differentials with poles, we retain the punctures in the cycle
calculation. Let $\ell_+$ be a large counterclockwise circle on $S_+$
in the $\lambda$-plane, thus a clockwise loop about $P_\infty^+$
in $z=1/\lambda$. The identities needed below are
\[
\rho_*a_j\equiv-a_j,\quad \iota_*a_j\equiv-a_j,\quad
\rho_*b_j\equiv b_j+\ell_+-\sum_{k\ne j}a_k,\quad
\sigma_*b_j\equiv-b_j+\sum_{k\ne j}a_k-\ell_+.
\]
The last two relations hold in
$H_1(\mathcal R\setminus\{P_\infty^+,P_\infty^-\},\Z)$.
Appendix~\ref{app:reflected-cycles} proves them with the auxiliary
cut loop $a_{g+1}$. For $d\Omega_2$, all deformations also avoid
$P_0^\pm$; loops about these points have zero integral.}\endgroup

\begin{figure}
\begin{tikzpicture}[scale=0.75, transform shape]
  
  
  
  \node at (2,1.8) {$\lambda_{g+1}$}; 
  \node at (2,-1.8) {$\bar\lambda_{g+1}$};
  \draw (2,1.5) -- (2,-1.5);
  \draw[-<-] (5,0) ellipse (0.8cm and 2.4cm);
  \node at (5,1.8) {$\lambda_{1}$}; 
  \node at (5,-1.8) {$\bar\lambda_{1}$};
  \draw (5,1.5) -- (5,-1.5);
  \node at (6.2,0.3) {$a_1$};
  \draw[-<-] (8,0) ellipse (0.8cm and 2.4cm);
  \node at (8,1.8) {$\lambda_{2}$}; 
  \node at (8,-1.8) {$\bar\lambda_{2}$};
  \draw (8,1.5) -- (8,-1.5);
  \node at (9.2,0.3) {$a_2$};
  
  \draw[ultra thick, loosely dotted](9,1.3) -- (11,1.3);

  \draw[-<-] (12,0) ellipse (0.8cm and 2.4cm);
  \node at (12,1.8) {$\lambda_{g}$}; 
  \node at (12,-1.8) {$\bar\lambda_{g}$};
  \draw (12,1.5) -- (12,-1.5);
  \node at (13.2,0.3) {$a_g$};
  
  \draw[-<-]
    (2,1) to[out=180,in=-90] (1,1.8)
    to[out=90,in=180] (3.5,2.5)
    to[out=0,in=90] (6,1.8)
    to[out=-90,in=0](5,1);
  \draw[dashed,-<-]
    (5,1) to[out=180,in=0](2,1);
    \node at (3.5,2.8) {$b_1$};
  \draw[-<-]
    (2,0.8) to[out=180,in=-90] (0.8,1.8)
    to[out=90,in=180] (5,3)
    to[out=0,in=90] (9,1.8)
    to[out=-90,in=0](8,0.8);
  \draw[dashed,-<-]
    (8,0.8) to[out=180,in=0](5,-3)
    to[out=180,in=0](2,0.8);
    \node at (4.5,3.3) {$b_2$};
  \draw[-<-]
    (2,0.6) to[out=180,in=-90] (0.6,1.8)
    to[out=90,in=180] (7,4)
    to[out=0,in=90] (13,1.8)
    to[out=-90,in=0](12,0.6);
  \draw[dashed,-<-]
    (12,0.6) to[out=180,in=0](8,-3.5)
    to(5,-3.5)
    to[out=180,in=0](2,0.6);
    \node at (6,4.3) {$b_g$};

  \node at (-1,0) {$\cdot$};
  \node at (-1,0.3) {$\infty$};


  \end{tikzpicture}
  \caption{\rev{A homology} basis for \rev{the} hyperelliptic Riemann surface $\mathcal{R}$}\label{fig:homologybasis}
\end{figure}

\subsection{Abelian differentials and Abel map}\label{subsec:differentials}
  Let \(\omega_1,\dots,\omega_g\) be holomorphic differentials on \(\mathcal{R}\), normalized by
\[
\int_{a_k} \omega_j = 2\pi i\, \delta_{kj}, \quad j,k = 1,\dots,g.
\]
Define the period matrix \(\tau = (\tau_{jk})\) by
\[
\tau_{jk} = \int_{b_k} \omega_j.
\]
The associated Riemann theta function is given by
\[
\theta(\mathbf{u}) = \sum_{\mathbf{n} \in \mathbb{Z}^g} 
\exp\left( \rev{\tfrac{1}{2} \mathbf{n}^{\mathsf T} \tau \mathbf{n} + \mathbf{n}^{\mathsf T}\mathbf{u}} \right).
\]
The series is absolutely convergent; this follows from the fact that \(\operatorname{Re}\tau\) is a negative-definite matrix. Moreover, if \(\mathbf{e}_k\) are the standard basis vectors of \(\C^g\) and \(\boldsymbol{\tau}_k = \tau \mathbf{e}_k\), for \(k=1,\dots,g\), then
\begin{equation}
    \theta(\mathbf{u} + 2\pi i \mathbf{e}_k) = \theta(\mathbf{u}),\quad \theta(\mathbf{u} + \boldsymbol{\tau}_k) = \exp(-\frac{1}{2}\tau_{kk}-u_k)\theta(\mathbf{u}). 
\end{equation}

\begingroup

With the normalization \(\int_{a_k}\omega_j=2\pi\ii\delta_{kj}\),
the pullback identity is \(\sigma^*\omega_j=-\overline{\omega_j}\).
Indeed, both sides have the same \(a\)-periods. Applying it to
\(\sigma_*b_k\equiv-b_k+\sum_{l\ne k}a_l\) gives
\(\overline{\tau_{jk}}=\tau_{jk}-2\pi\ii(1-\delta_{jk})\).
Hence the period matrix satisfies
\endgroup
\begin{equation}\label{eq:taucondition}
\Im \tau_{jk} =
\begin{cases}
\pi & \text{if } j\neq k, \\[0.3ex]
0   & \text{if } j = k.
\end{cases}
\end{equation}

\rev{It follows that the Riemann theta function satisfies the conjugation symmetry:}
\begin{equation}\label{eq:thetaconjugate}
\overline{\theta(\mathbf{u})} = \theta(\bar{\mathbf{u}}).
\end{equation}
    Let \(\Gamma\) be the lattice in \(\mathbf{C}^g\) generated by the columns of the matrix \((2\pi i I \mid \tau)\). The Jacobian variety of \(\mathcal{R}\) is defined as
\[
\operatorname{Jac}(\mathcal{R}) = \mathbf{C}^g / \Gamma.
\]
The map
\[
\Ab_Q(P) = \left( \int_Q^P \omega_1, \dots, \int_Q^P \omega_g \right)^{\rev{\mathsf T}} \bmod \Gamma,
\]
is called the \textbf{Abel-Jacobi map}.

\Needspace{5\baselineskip}
\subsection{Baker--Akhiezer function and theta representation}
\label{subsec:BA-theta}
\subsubsection*{\rev{Normalized differentials}}

\begingroup

Put \(Q\) equal to the branch point with \(\lambda(Q)=\bar\lambda_{g+1}\) and \(\mu(Q)=0\), and set
\[
c=\sum_j(\lambda_j+\bar\lambda_j),\qquad
d=\sum_j(\lambda_j^{-1}+\bar\lambda_j^{-1}).
\]
For a meromorphic differential \(\eta\), write
\(\mathcal N\eta=\eta-\sum_k(\int_{a_k}\eta)\omega_k/(2\pi\ii)\).
The differentials used here are
\begin{align*}
d\Omega_1&=\mathcal N\left[\frac{\lambda^{g+1}-(c/2)\lambda^g}{\mu}\,d\lambda\right],\\
d\Omega_2&=\mathcal N\left[-\frac{\mu_0}{\mu}
 \left(\lambda^{-2}-\frac d2\lambda^{-1}\right)d\lambda\right],\\
d\Omega_3&=\mathcal N\left[\frac{\lambda^g}{\mu}\,d\lambda\right].
\end{align*}
Define \(\Omega_j(P)=\int_Q^P d\Omega_j\), \(j=1,2,3\),
using the common lifted paths specified below.
In particular \(\int_{a_k}d\Omega_j=0\) for all \(j,k\).
The terms involving \(c\) and \(d\) cancel the residue terms in the
respective meromorphic differentials at infinity and zero.
Thus \(d\Omega_1,d\Omega_2\)
have zero residues. The residues of \(d\Omega_3\) at \(P_\infty^+,P_\infty^-\)
are \(-1,+1\); consequently \(\int_{\ell_+}d\Omega_3=2\pi\ii\).

\subsubsection*{\rev{Common paths and reality}}
With the cycles and sheet conventions fixed above, uniqueness of the
normalized differentials gives
\[
\rho^*\omega_j=\overline{\omega_j},\qquad
\rho^*d\Omega_m=\overline{d\Omega_m},\qquad
\iota^*\omega_j=-\omega_j,\quad \iota^*d\Omega_m=-d\Omega_m.
\]
In particular, putting $U_j=\int_{b_j}d\Omega_1$ and
$V_j=\int_{b_j}d\Omega_2$, one obtains
\[
U_j,V_j\in\R,\qquad
\overline{\int_{b_j}d\Omega_3}=\int_{b_j}d\Omega_3+2\pi\ii.
\]
Here the puncture loop contributes $2\pi\ii$ because the residue at
$P_\infty^+$ is $-1$; loops about $P_0^\pm$ contribute zero.

Let $\xi$ run from $Q$ to $\bar Q$ along the right bank of the omitted
cut on $S_+$. For an $\iota$-odd differential $\eta$,
$2\int_\xi\eta=\int_{a_{g+1}}\eta$. Hence
\[
\int_\xi d\Omega_1=\int_\xi d\Omega_2=0,\qquad
\int_\xi d\Omega_3=\pi\ii,\qquad \int_\xi\omega_j=-\pi\ii.
\]
Choose a positive real $R$ to the right of the cuts. A path $\gamma_R$
from $Q$ to the point over $R$ on $S_+$ is taken below all cuts and
below the lower $b$-arcs. In relative homology,
\[
\gamma_R-\rho\gamma_R\equiv\xi+\sum_{k=1}^g a_k.
\]
The bounded-region argument in Appendix~\ref{app:reflected-cycles}
establishes this relation, including its signs and the absence of an
infinity-loop term.
Using the pullback identities above, integration gives
\[
2\ii\Im\int_{\gamma_R}d\Omega_1
=2\ii\Im\int_{\gamma_R}d\Omega_2=0,\qquad
2\ii\Im\int_{\gamma_R}d\Omega_3=\pi\ii,
\quad
2\ii\Im\int_{\gamma_R}\omega_j=-\pi\ii+2\pi\ii=\pi\ii.
\]
Thus $\Omega_1,\Omega_2$ are real, $\Im\Omega_3=\pi/2$, and
$\Im\int_{\gamma_R}\omega_j=\pi/2$ at the endpoint.
Continue all integrals together from this endpoint along
$(x,+\sqrt{\prod_j(x-\lambda_j)(x-\bar\lambda_j)})$, $x>0$.
These identities persist because $\rho$ fixes that arc. Crossing a
projected cut changes its sheet label but not the continued point.

Extend $\gamma_R$ from its endpoint to the right along the positive real
axis, without crossing a cut, to obtain $\gamma_\infty:Q\to P_\infty^+$.
Write $\Ab_-:=\Ab_{P_\infty^-}$ and
$\mathbf r=2\int_{\gamma_\infty}\omega$, where
$\omega=(\omega_1,\ldots,\omega_g)^{\mathsf T}$.
The fixed path $P_\infty^-\to Q$ is $-\iota\gamma_\infty$, so
\[
\Ab_-(P)=\frac{\mathbf r}{2}+\int_Q^P\omega,
\qquad \Omega_m(P)=\int_Q^P d\Omega_m.
\]
The $Q\to P$ part is the same lifted path in all four expressions;
any continuation around a cycle is made in all of them together.
Pair the paths to $P$ and $\iota P$ by $\iota$, so
$\Omega_m(\iota P)=-\Omega_m(P)$ for these lifts.
Write $\mathbf U=(U_j)$ and $\mathbf V=(V_j)$.
\rev{\begin{lemma}[Reality and asymptotic constants]\label{lem:real-data}
For the normalized differentials and common paths fixed above,
$\mathbf U,\mathbf V\in\R^g$ and
$\overline{\mathbf r}-\mathbf r=-2\pi\ii(1,\ldots,1)^{\mathsf T}$.
There exist $E,H\in\R$ and $\beta>0$ such that, with the real logarithm
on the positive real arc,
\begin{align*}
\Omega_1(P)&=\pm(\lambda-E/2+O(\lambda^{-1})),\\
\Omega_2(P)&=\pm(H/2+O(\lambda^{-1})),\\
\Omega_3(P)&=\pm(\log\lambda+\pi\ii/2-\tfrac12\log\beta
                         +O(\lambda^{-1})),\qquad P\to P_\infty^\pm,
\end{align*}

Moreover $\Im\Omega_3=\pi/2$ on that arc.
\end{lemma}
\begin{proof}
The reflected-cycle and relative-path calculations above establish
the real $b$-periods and the imaginary parts of the integrals on the
positive real arc. Continuing to infinity gives the displayed
expansions. Their constants are
\[
E=-2\lim_{P\to P_\infty^+}(\Omega_1-\lambda),\qquad
H=2\Omega_2(P_\infty^+),\qquad
\beta=\exp\!\left(-2\Re\lim_{P\to P_\infty^+}(\Omega_3-\log\lambda)\right).
\]
At zero the corresponding expansions, with $\epsilon=\pm1$, are
\[
\Omega_1=k_1^\epsilon+O(\lambda),\quad
\Omega_2=\epsilon\lambda^{-1}+k_2^\epsilon+O(\lambda),\quad
\Omega_3=k_3^\epsilon+O(\lambda).
\]
The constants are fixed by the same paths, need not vanish, and satisfy
$k_j^-=-k_j^+$; $k_1^+,k_2^+\in\R$ and $\Im k_3^+=\pi/2$.

The chosen paths give $\mathbf r=\Ab_-(P_\infty^+)$.
As detailed in Appendix~\ref{app:reflected-cycles}, the path
$-\iota\gamma_\infty+\gamma_\infty$ lies inside a canonical cut
polygon for the chosen cycle representatives. The Abel primitive on
this polygon uses precisely this path at $P_\infty^+$.
The bilinear relation for $\omega_j,d\Omega_3$, with residues
$-1,+1$ at $P_\infty^+,P_\infty^-$, therefore gives the actual values
\[
2\pi\ii\int_{b_j}d\Omega_3
 =2\pi\ii[-\Ab_{-,j}(P_\infty^+)+\Ab_{-,j}(P_\infty^-)],
\]
and thus
\[
r_j=-\int_{b_j}d\Omega_3,\qquad
\Im r_j=\pi,\qquad \bar{\mathbf r}-\mathbf r=-2\pi\ii(1,\ldots,1)^{\mathsf T}.
\]
The first equality uses the common polygon determination, not merely
congruence modulo periods; its sign agrees with
\cite[Sections~5.1--5.3]{CaliniIvey}.
\end{proof}}
\endgroup
\subsubsection*{\rev{Divisor data}}
\rev{For a divisor $\mathcal E$ on $\mathcal R$, let
\[
  L(\mathcal E):=\{f\text{ meromorphic on }\mathcal R\mid
                  (f)+\mathcal E\ge0\}\cup\{0\}.
\]}

\rev{Let $K_-=(K_{-,1},\dots,K_{-,g})^{\mathsf T}$ be the Riemann constant}
with base point $P_\infty^-$:
\[
\rev{K_{-,j} = \pi\ii+\frac{1}{2}\tau_{jj}
-\frac{1}{2\pi\ii}\sum_{k\neq j}\int_{a_k}\omega_k(P)\int_{P_\infty^-}^P\omega_j,}
\quad j=1,\dots,g.
\]

\rev{Let $\delta=P_1+\cdots+P_g$ be a positive divisor of degree $g$.
Choose a representative $\mathbf D\in\C^g$ such that}
\[
  \rev{\mathbf D\equiv\Ab_-(\delta)+K_-\pmod\Gamma.}
\]
\rev{Assume the following conditions:}
\begin{enumerate}[label=\textup{(D\arabic*)},leftmargin=2.3em]
\item $\dim_{\C} L(\delta)=1$,
\item $P_j\notin\{P_\infty^\pm,P_0^\pm\}$ for $j=1,\dots,g$,
\item \rev{the chosen representative satisfies \(\overline{\mathbf D}=-\mathbf D\).}
\end{enumerate}
Set
\[
  \rev{\mathbf{W}(s,t)=-\frac{\ii}{2}\bigl(s\,\mathbf{U}+t\,\mathbf{V}\bigr).}
\]

\begingroup

Following the reality argument in \cite[Section~5.2]{CaliniIvey},
we use this purely imaginary representative in all theta formulas and
in the normalization of $\alpha$ below. Condition (D3) is an assumption
on the selected data, not a consequence of specifying $\delta$ alone.
To preserve the same potential, a quasiperiod shift of $\mathbf D$ may
require adjusting the prefactor in~\eqref{eq:q-theta}. To choose data, start with
\(\mathbf D\in\ii\R^g\) and solve the Jacobi inversion problem: the zero
divisor of \(\theta(\Ab_-(P)-\mathbf D)\) is \(\delta\).
Retain choices for which (D1)--(D2) hold and
\(\theta(\mathbf D)\theta(\mathbf D-\mathbf r)\ne0\).
Work on a parameter domain containing \((0,0)\) where
\(\theta(\mathbf W+\mathbf D)\ne0\).
\rev{Take this domain connected. The BA frame and scaled Sym formula
define a smooth map on it. Its $s$-curve Frenet description applies on
$q\ne0$, whereas it is a surface immersion on $q_t\ne0$.
The real Lax coefficient $R$ is constructed without division by $q$,
so the smooth map can extend across zeros of $q$.}
Admissible purely imaginary $\mathbf D$ exist: for each marked point
the holomorphic function $\theta(\Ab_-(P)-\mathbf D)$ cannot vanish
on an open subset of $\ii\R^g$: applying the identity theorem in each
variable would make it identically zero. The finite union of these
zero sets has empty interior. Avoiding them gives a degree-$g$
zero divisor disjoint from the marked points; $\theta(\mathbf D)\ne0$
implies its nonspeciality by Riemann's vanishing theorem (as in the
quotient argument below). This ensures $\theta(\mathbf W+\mathbf D)\ne0$ in a neighborhood of
the origin, without asserting global nonvanishing in higher genus.
In genus one this inversion is given explicitly in Section~\ref{sec:genus1}.

The relation
$\Ab_-(\sigma P)\equiv\mathbf r-\overline{\Ab_-(P)}\pmod\Gamma$,
theta conjugation, $\overline{\mathbf r}-\mathbf r\in2\pi\ii\Z^g$,
and (D3) show that the local expressions
\[
\overline{\theta(\Ab_-(\sigma P)-\mathbf D)},\qquad
\theta(\Ab_-(P)-\mathbf D-\mathbf r)
\]
differ only by nowhere-zero holomorphic factors. The second theta
section therefore has zero divisor $\sigma\delta$, as used in the
proof of Theorem~\ref{thm:representation}. In particular,
$\sigma\delta\sim\delta+P_\infty^+-P_\infty^-$, where $\sim$ denotes
linear equivalence; this is a consequence, not an additional assumption.
\endgroup
\begingroup
\subsubsection*{Theta expressions and scalar vanishing}
For $\alpha\in\C^\times$, define the two theta expressions
\begin{align}
\label{eq:tildepsi1}
\tilde{\psi}_1(P;s,t)
:=&
-i\alpha\sqrt{\beta}\,
\exp\!\left[
 \tfrac{\ii s}{2}\!\left(\Omega_1(P)+\tfrac{E}{2}\right)
+\tfrac{\ii t}{2}\!\left(\Omega_2(P)-\tfrac{H}{2}\right)
+\Omega_3(P)\right]\!\\
&\times
\frac{\theta(\Ab_-(P)-\mathbf{W}(s,t)-\mathbf{D}-\mathbf{r})\,\theta(\mathbf{D}-\mathbf{r})}
     {\theta(\Ab_-(P)-\mathbf{D})\,\theta(\mathbf{W}(s,t)+\mathbf{D})},\nonumber\\
\label{eq:tildepsi2}
\tilde{\psi}_2(P;s,t)
:=&
\exp\!\left[
 \tfrac{\ii s}{2}\!\left(\Omega_1(P)-\tfrac{E}{2}\right)
+\tfrac{\ii t}{2}\!\left(\Omega_2(P)+\tfrac{H}{2}\right)\right]\!
\frac{\theta(\Ab_-(P)-\mathbf{W}(s,t)-\mathbf{D})\,\theta(\mathbf{D})}
     {\theta(\Ab_-(P)-\mathbf{D})\,\theta(\mathbf{W}(s,t)+\mathbf{D})}.
\end{align}

\begin{lemma}[Scalar vanishing]\label{lem:scalar-vanishing}
On the parameter domain specified above, suppose $f$ is a single-valued
scalar meromorphic function away from the marked points, with pole
divisor bounded by $\delta$ and essential factors
$e^{\pm\ii s/(2z)}$ at $P_\infty^\pm$ and
$e^{\pm\ii t/(2\lambda)}$ at $P_0^\pm$.
After removing the essential exponentials, suppose its holomorphic
factors are $O(1)$ at $P_\infty^+$, $O(z)$ at $P_\infty^-$, and
$O(1)$ at $P_0^\pm$. Then $f=0$.
\end{lemma}
\begin{proof}
Let $\phi$ be the scalar theta expression in~\eqref{eq:tildepsi2}.
The fixed divisor $\delta$ bounds its poles. Let $\delta_{s,t}$
be the moving zero divisor of
$\theta(\Ab_-(P)-\mathbf W-\mathbf D)$ before cancellation.
Since $\theta(\mathbf W+\mathbf D)\ne0$, Riemann's theorem gives
$\deg\delta_{s,t}=g$, $P_\infty^-\notin\delta_{s,t}$, and
$\Ab_-(\delta_{s,t})+K_-\equiv\mathbf W+\mathbf D$.
Also $\dim L(\delta_{s,t})=1$. Otherwise, subtracting its value at
$P_\infty^-$ from a nonconstant member would give a nonzero function
vanishing there. Thus $\delta_{s,t}-P_\infty^-$ would be linearly
equivalent to an effective divisor of degree $g-1$. Since
$\Ab_-(P_\infty^-)=0$, Abel's theorem and Riemann's vanishing theorem
would imply $\theta(\mathbf W+\mathbf D)=0$, a contradiction.
For $g=0$ the assertion is immediate. In particular,
$L(\delta_{s,t})$ consists only of constants.

The $b_k$ multiplier of the theta quotient in $\phi$ is $e^{W_k}$,
which cancels its exponential multiplier $e^{-W_k}$; the $a$ multipliers
are one. Hence $\phi$ is single-valued. Evenness of theta and the
expansions in Lemma~\ref{lem:real-data} give
$e^{\ii s/(2z)}\phi=1+O(z)$ at $P_\infty^-$. The essential factors
cancel in $f/\phi$. At a point with
multiplicities $d$ in $\delta$ and $m$ in $\delta_{s,t}$,
$\operatorname{ord}\phi=m-d$ and
$\operatorname{ord}f\ge-d$, so the quotient has order at least $-m$.
The same bound holds at marked points after removing the exponentials.
Thus $f/\phi\in L(\delta_{s,t})$, even if numerator and
denominator zeros cancel. This quotient is constant and vanishes at
$P_\infty^-$, so $f=0$.

\end{proof}

\endgroup

\Needspace{8\baselineskip}
\begin{lemma}[Normalized finite-gap Baker--Akhiezer function]
\label{lem:BA-PLR}
\rev{Fix the marked spectral curve, homology basis, local coordinates, common integration paths, divisor \(\delta\), and a nonzero constant \(\alpha\). On the connected parameter domain just specified, there exists a unique vector-valued function}
\[
  \tilde\psi(P;s,t)=\bigl(\tilde\psi_1(P;s,t),\tilde\psi_2(P;s,t)\bigr)
\]
such that for each \rev{$(s,t)$ in this domain} the following hold:
\begin{enumerate}[label=\textup{(BA\arabic*)},leftmargin=2.3em]
\item $\tilde\psi_j(\,\cdot\,;s,t)$ is meromorphic on
      $\mathcal{R}\setminus\{P_\infty^\pm,P_0^\pm\}$ and its pole divisor is bounded by~$\delta$.
\item Let $z=\lambda^{-1}$ be the local parameter near $P_\infty^\pm$.
      Then
      \[
      \tilde\psi \sim \frac{\alpha}{z}\bigl[(1,0)+O(z)\bigr]e^{\ii s/(2z)}\ (P\to P_\infty^+),\qquad
      \tilde\psi \sim \bigl[(0,1)+O(z)\bigr]e^{-\ii s/(2z)}\ (P\to P_\infty^-),
      \]
      and $\tilde\psi=O(1)\exp(\pm \ii t/(2\lambda))$ ($P\to P_0^\pm$).
      \rev{After removal of these exponentials and, at $P_\infty^+$,
      multiplication by $z$, the resulting local factors are holomorphic
      in the local coordinate and analytic in $(s,t)$.}
\end{enumerate}
\rev{The function is given by~\eqref{eq:tildepsi1}--\eqref{eq:tildepsi2}, with the constants fixed in Lemma~\ref{lem:real-data}.}
\end{lemma}

\begin{proof}
\begingroup\sloppy
\emph{(Existence.)}
Since $\dim L(\delta)=1$, Riemann's vanishing theorem implies that
$\theta(\Ab_-(P)-\mathbf{D})$ vanishes precisely at the points of $\delta$.
Hence $1/\theta(\Ab_-(P)-\mathbf{D})$ has poles bounded by $\delta$ and no other poles.
All other theta factors are holomorphic in $P$,
so \textup{(BA1)} holds.
\par\endgroup

\begingroup
For \textup{(BA2)}, $\Ab_-(P)\to\mathbf r$ at $P_\infty^+$, so evenness of
$\theta$ makes the theta quotient in $\tilde\psi_1$ tend to one.
At $P_\infty^-$, $\Ab_-(P)\to0$, and the quotient in $\tilde\psi_2$
likewise tends to one. The fixed branch of $\Omega_3$ gives
\[
e^{\Omega_3(P)}=
\begin{cases}
\ii\beta^{-1/2}z^{-1}(1+O(z)),&P\to P_\infty^+,\\
-\ii\beta^{1/2}z(1+O(z)),&P\to P_\infty^-.
\end{cases}
\]
Combining these with the constant terms of $\Omega_1,\Omega_2$ yields
\[
\begin{aligned}
e^{-\ii s/(2z)}\tilde\psi
&=(\alpha z^{-1}+O(1),\,O(1))&& (P\to P_\infty^+),\\
e^{\ii s/(2z)}\tilde\psi
&=(O(z),\,1+O(z))&& (P\to P_\infty^-).
\end{aligned}
\]
At $P_0^\pm$, $\Omega_1,\Omega_3$ are regular and
$\Omega_2=\pm\lambda^{-1}+O(1)$. By (D2), the theta denominators do not
vanish there, so $e^{\mp\ii t/(2\lambda)}\tilde\psi$ is holomorphic.
The local factors are analytic in $(s,t)$ wherever
$\theta(\mathbf W+\mathbf D)\ne0$.
\endgroup

\begingroup

\emph{(Single-valuedness.)}
Adding \(b_k\) to the common path changes \(\Ab_-\) by \(\tau\mathbf e_k\),
\(\Omega_1,\Omega_2,\Omega_3\) by \(U_k,V_k,-r_k\).
In \eqref{eq:tildepsi1} the theta quotient acquires
\(e^{W_k+r_k}\), while its exponential acquires \(e^{-W_k-r_k}\).
For \eqref{eq:tildepsi2} these factors are \(e^{W_k}\) and \(e^{-W_k}\).
The \(a\)-cycle multipliers are one. Thus the normalized BA functions are
single-valued under common continuation.

\emph{(Uniqueness.)} The difference of two normalized BA vectors has
\rev{the local orders required by Lemma~\ref{lem:scalar-vanishing}.
Applying that lemma to each component proves uniqueness.} This is the quotient argument of
\cite[Section~2.7]{Belokolos}.
\endgroup
\end{proof}
\subsubsection*{\rev{The Lax equations}}
\begin{proposition}[The normalized BA function solves the PLR Lax pair]
\label{prop:tildepsi-Lax}
Let $\tilde\psi$ be the Baker--Akhiezer function in Lemma~\ref{lem:BA-PLR}, \rev{with
\(\alpha=\theta(\mathbf D)/(\sqrt\beta\,\theta(\mathbf D-\mathbf r))\),} and define
\begin{equation}\label{eq:q-theta}
  q(s,t)
   = 2\ii\sqrt{\beta}\,
     \rev{\exp(-\ii Es/2 + \ii H t/2)}\,
     \frac{\theta\bigl(\mathbf{W}(s,t)+\mathbf{D}-\mathbf{r}\bigr)}
          {\theta\bigl(\mathbf{W}(s,t)+\mathbf{D}\bigr)}.
\end{equation}
\rev{On $\mathcal R\setminus\{P_\infty^\pm,P_0^\pm\}$, with
$\lambda=\lambda(P)$, the row vector $\tilde\psi(P;s,t)$ satisfies
the following equations as meromorphic identities in $P$:}
\begin{equation}\label{eq:tildepsi-Lax}
  \partial_s \tilde\psi = \tilde\psi\,L,\qquad
  \partial_t \tilde\psi = \tilde\psi\,M,
\end{equation}
where $L,M$ are the matrices in \eqref{eq:L-M-PLR}.
\rev{In $M$, $\Re(q_{st}/q)$ is understood as the real analytic coefficient
$R$ constructed below, satisfying $q_{st}=Rq$ throughout the parameter
domain and hence agreeing with $q_{st}/q$ wherever $q\ne0$.}
In particular, $q$ is a
\rev{finite-gap solution of the local system in Proposition~\ref{prop:Lax-representation}, and hence of \eqref{eq:PLR-q} with the specified real integration constant.}
\end{proposition}

\begin{proof}
\begingroup
At infinity the normalized local factors are
\begin{align*}
\tilde\psi&=\alpha e^{\ii s/(2z)}
 (z^{-1}+\alpha_{1+}+O(z),\,\alpha_{2+}+O(z))
 &&(P_\infty^+),\\
\tilde\psi&=e^{-\ii s/(2z)}
 (\alpha_{1-}z+O(z^2),\,1+\alpha_{2-}z+O(z^2))
 &&(P_\infty^-).
\end{align*}
Set
\[
\widetilde L=\begin{pmatrix}\ii\lambda/2&\ii\alpha_{2+}\\
 -\ii\alpha_{1-}&-\ii\lambda/2\end{pmatrix},\qquad
f=\tilde\psi_s-\tilde\psi\widetilde L.
\]
At $P_\infty^+$ put $p=e^{-\ii s/(2z)}\tilde\psi$. Then
\[
e^{-\ii s/(2z)}f=(p_{1s}+\ii\alpha_{1-}p_2,
 p_{2s}+\ii p_2/z-\ii\alpha_{2+}p_1)=O(1).
\]
The pole cancels since $p_1=\alpha/z+O(1)$ and
$p_2=\alpha\alpha_{2+}+O(z)$. At $P_\infty^-$ put
$p=e^{\ii s/(2z)}\tilde\psi$. Here
\[
e^{\ii s/(2z)}f=(p_{1s}-\ii p_1/z+\ii\alpha_{1-}p_2,
 p_{2s}-\ii\alpha_{2+}p_1)=O(z),
\]
because $p_1=\alpha_{1-}z+O(z^2)$ and $p_2=1+O(z)$.
At zero the essential factors are independent of $s$ and
$\widetilde L$ is holomorphic. At a point of multiplicity $m$ in the
fixed divisor $\delta$, a local expansion starts with $w^{-m}$;
parameter differentiation changes its coefficients but not the lowest
power of $w$. Thus the pole bound remains $\delta$.
Together with the local orders just checked, this permits application
of \rev{Lemma~\ref{lem:scalar-vanishing}} to each
component of $f$, giving $f=0$.

For the prescribed $\alpha$, expansion of the theta formulas gives
\begin{align*}
\alpha_{2+}&=\sqrt\beta e^{-\ii Es/2+\ii Ht/2}
 \frac{\theta(\mathbf W+\mathbf D-\mathbf r)}{\theta(\mathbf W+\mathbf D)},\\
\alpha_{1-}&=-\sqrt\beta e^{\ii Es/2-\ii Ht/2}
 \frac{\theta(\mathbf W+\mathbf D+\mathbf r)}{\theta(\mathbf W+\mathbf D)}.
\end{align*}
Theta conjugation, evenness, $\bar{\mathbf W}=-\mathbf W$,
$\bar{\mathbf D}=-\mathbf D$, and
$\bar{\mathbf r}-\mathbf r\in2\pi\ii\Z^g$ imply
$\alpha_{1-}=-\overline{\alpha_{2+}}$. Thus
$2\ii\alpha_{2+}=q$, $2\ii\alpha_{1-}=\bar q$, and $\widetilde L=L$.

At $P_0^\epsilon$, retaining the finite parts of all Abelian integrals,
write
\[
\tilde\psi=e^{\epsilon\ii t/(2\lambda)}
 (h_0^\epsilon+h_1^\epsilon\lambda+O(\lambda^2)),\qquad\epsilon=\pm1.
\]
The matrix $\mathcal H$ with rows $h_0^+,h_0^-$ is invertible;
the argument is the one in \cite[Proposition~4.2]{Date}, adapted to
our normalization. If $\mathcal H v=0$, the local factors of
$\tilde\psi v$ vanish to order at least one at both zero points.
Consequently $\tilde\psi v/\lambda$ has
holomorphic factors $O(1)$ at both zero points, $O(1)$ at
$P_\infty^+$, and $O(z)$ at $P_\infty^-$.
The vanishing assertion makes it zero, and the two infinity
normalizations imply $v_1=v_2=0$. Define
\[
C=\mathcal H^{-1}\operatorname{diag}(\ii/2,-\ii/2)\mathcal H.
\]
The coefficients $h_0^\pm$ are analytic in $(s,t)$, so the invertibility
just proved makes $C$ analytic on the specified parameter domain.
Then $h_0^\epsilon C=\epsilon\ii h_0^\epsilon/2$ and
\begin{align*}
e^{-\epsilon\ii t/(2\lambda)}
 (\tilde\psi_t-\tilde\psi C/\lambda)
={}&\lambda^{-1}h_0^\epsilon(\epsilon\ii I/2-C)
 +(h_0^\epsilon)_t\\
 &+h_1^\epsilon(\epsilon\ii I/2-C)+O(\lambda)=O(1).
\end{align*}
At infinity the leading normalizations are independent of $t$;
the residual has factors $O(1)$ at $P_\infty^+$ and $O(z)$ at
$P_\infty^-$, since $1/\lambda=z$. Its other poles remain bounded
by $\delta$. The same vanishing assertion yields
$\tilde\psi_t=\tilde\psi C/\lambda$, with
$\operatorname{tr}C=0$ and $\det C=1/4$.

To obtain matrix compatibility, stack the rows at $P$ and $\iota P$.
For large $\lambda$ their determinant is $\alpha\lambda+O(1)$,
so the matrix is invertible and
$L_t-(C/\lambda)_s+[C/\lambda,L]=0$.
Writing $C=\left(\begin{smallmatrix}d&h\\k&-d\end{smallmatrix}\right)$
gives, respectively, the constant and $\lambda^{-1}$ coefficients
\[
\begin{pmatrix}0&q_t/2-\ii h\\-\bar q_t/2+\ii k&0\end{pmatrix}=0,
\qquad
\begin{pmatrix}
-d_s-(h\bar q+qk)/2&dq-h_s\\
d\bar q-k_s&d_s+(h\bar q+qk)/2
\end{pmatrix}=0.
\]
Hence
\[
h=-\tfrac\ii2 q_t,\quad k=-\tfrac\ii2\bar q_t,\quad
d_s=\tfrac\ii4(|q|^2)_t,\quad
-\tfrac\ii2q_{st}=dq,\quad-\tfrac\ii2\bar q_{st}=d\bar q.
\]
The last two equations imply $(d+\bar d)\bar q=0$.
The potential is real analytic on the connected parameter domain and
$q(0,0)=2\ii\sqrt\beta\,\theta(\mathbf D-\mathbf r)/\theta(\mathbf D)\ne0$.
Its nonzero set is therefore dense, and $d+\bar d=0$ throughout the
domain. With $R=2\ii d\in\R$ we obtain
\[
C=\frac\ii2\begin{pmatrix}-R&-q_t\\-\bar q_t&R\end{pmatrix},\qquad
q_{st}=Rq,\quad R_s=-\tfrac12(|q|^2)_t,\quad R^2+|q_t|^2=1.
\]
The last equality follows from $\det C=1/4$ and reflects the chosen
principal parts at zero. No division by $q$ is used, so $R$ and the
Lax pair extend through its zeros. The real integration constant in
Section~\ref{subsec:LR} then gives~\eqref{eq:PLR-q}.
\endgroup
\end{proof}

\begin{corollary}[Gauge-fixed representative]
\label{cor:BA-theta}
\rev{Assume
\(\alpha=\theta(\mathbf D)/(\sqrt\beta\,\theta(\mathbf D-\mathbf r))\)
as in Proposition~\ref{prop:tildepsi-Lax}.}
Define
\begin{equation}\label{eq:psi-from-tildepsi}
  \psi(P;s,t)
  := \frac{\theta(\Ab_-(P)-\mathbf D)}{\theta(\mathbf D)}\,\tilde\psi(P;s,t),
\end{equation}
where \(\tilde\psi\) is defined by \eqref{eq:tildepsi1}, \eqref{eq:tildepsi2}.
Then $\psi=(\psi_1,\psi_2)$ has the following theta-functional expression
\begin{align}
\label{eq:psi}
\psi_1(P;s,t)
&= -\ii\exp\!\left[
 \tfrac{\ii s}{2}\!\left(\Omega_1(P)+\tfrac{E}{2}\right)
+\tfrac{\ii t}{2}\!\left(\Omega_2(P)-\tfrac{H}{2}\right)
+\Omega_3(P)\right]\nonumber\\
&\quad\mathbin{\rev{\times}}\frac{\theta(\Ab_-(P)-\mathbf W(s,t)-\mathbf D-\mathbf r)}
     {\theta(\mathbf W(s,t)+\mathbf D)},\\
\psi_2(P;s,t)\nonumber
&= \exp\!\left[
 \tfrac{\ii s}{2}\!\left(\Omega_1(P)-\tfrac{E}{2}\right)
+\tfrac{\ii t}{2}\!\left(\Omega_2(P)+\tfrac{H}{2}\right)\right]
\frac{\theta(\Ab_-(P)-\mathbf W(s,t)-\mathbf D)}
     {\theta(\mathbf W(s,t)+\mathbf D)}.
\end{align}
Moreover, $\psi$ satisfies the Lax system satisfied by $\tilde\psi$ in \eqref{eq:tildepsi-Lax}, i.e.,
\begin{equation}\label{eq:psi-Lax}
  \partial_s\psi=\psi\,L,\qquad
  \partial_t\psi=\psi\,M,
\end{equation}
with the potential $q$ given by \eqref{eq:q-theta}.
\end{corollary}

\begin{proof}
Substituting $\alpha=\frac{\theta(\mathbf D)}{\sqrt{\beta}\,\theta(\mathbf D-\mathbf r)}$ into \eqref{eq:tildepsi1} and multiplying by
$\theta(\Ab_-(P)-\mathbf D)/\theta(\mathbf D)$ gives \eqref{eq:psi-from-tildepsi} and the
explicit formulas \eqref{eq:psi} immediately.

Since the prefactor in \eqref{eq:psi-from-tildepsi} depends only on $P$ and is independent
of $(s,t)$, differentiating with respect to $s$ and $t$ shows that $\psi$ satisfies exactly the
same equations as $\tilde\psi$. Hence \eqref{eq:psi-Lax} follows from
Proposition~\ref{prop:tildepsi-Lax}.
\end{proof}

\begin{remark}
\begingroup

The normalized functions \(\tilde\psi_j\) are single-valued by the common-path
cancellation in Lemma~\ref{lem:BA-PLR}. Multiplication by the theta section
in \eqref{eq:psi-from-tildepsi} gives the gauge-fixed \(\psi\), with a common
nonzero scalar transition factor depending only on \(P\).
For spectral differentiation, fix a simply connected neighborhood of the
chosen real reconstruction point and use the common lifted paths of this
\rev{subsection. Across a cut, continue all factors together; the common
transition factor specifies the change of the chosen representative.}
The resulting normalized frames in overlapping branches differ by a
diagonal \(\SU\) left factor depending only on \(\lambda\); their reconstructed
surfaces are related by a fixed rotation and translation at the chosen
reconstruction point.
\endgroup
\end{remark}

\begingroup

\subsection{Representation formula}
\begin{theorem}\label{thm:representation}
Let $\psi$ be the gauge-fixed BA function in Corollary~\ref{cor:BA-theta}
on the connected parameter domain specified in Section~\ref{subsec:BA-theta}, and let $P_*$ lie
over $\Lambda_0>0$. Along a real spectral neighborhood use the linked
paths fixed in Section~\ref{subsec:BA-theta}. Then
$\rev{\varrho}=|\psi_1|^2+|\psi_2|^2>0$, and
\[
\Psi=\rev{\varrho}^{-1/2}\begin{pmatrix}\psi_1&\psi_2\\
 -\overline{\psi_2}&\overline{\psi_1}\end{pmatrix},\qquad
\gamma=\left.\Lambda_0(\partial_\lambda\Psi)\Psi^{-1}\right|_{P_*}
\]
solve the Lax pair and~\eqref{eq:LR}, respectively.
The first fundamental form is
\begin{equation}\label{eq:immersion-metric}
\langle d\gamma,d\gamma\rangle
=\Lambda_0^2ds^2+2R\,ds\,dt+\Lambda_0^{-2}dt^2,
\qquad \det(\langle\gamma_i,\gamma_j\rangle)=|q_t|^2.
\end{equation}
We consider the reconstructed surface only on nonempty open subsets
where $q_t\ne0$, so it is an immersion. Its $s$-speed is $\Lambda_0$;
Frenet curvature and torsion are used only where $q\ne0$.

At $P_*$, its matrix entries are
\begin{align*}
\gamma_{11}&=\frac{\Lambda_0}{2\rev{\varrho}}\sum_{j=1}^2
 (\overline{\psi_j}\,\partial_\lambda\psi_j
       -\psi_j\overline{\partial_\lambda\psi_j}),\\
\gamma_{21}&=\frac{\Lambda_0}{\rev{\varrho}}
 (\overline{\psi_2}\,\overline{\partial_\lambda\psi_1}
       -\overline{\psi_1}\,\overline{\partial_\lambda\psi_2}).
\end{align*}
Conjugates here are actual complex conjugates; derivatives are along
the real spectral coordinate. Under~\eqref{eq:R3-su2}, the Cartesian
coordinate vector is
$(-2\ii\gamma_{11},\,2\Im\gamma_{21},\,2\Re\gamma_{21})$.
\end{theorem}
\begin{proof}
At $(s,t)=(0,0)$, away from the marked points
$P_\infty^\pm,P_0^\pm$, the two components have zero divisors
$\sigma\delta$ and $\delta$, respectively, by (D3) and the theta
conjugation identity. If both vanished at a real spectral point $P$,
then $\delta$ would contain both $P$ and $\iota P$, since
$\sigma P=\iota P$. The meromorphic function
$1/(\lambda-\lambda(P))$ would belong to $L(\delta)$ and be
nonconstant, contradicting (D1). There are no real branch points.
Thus $\rev{\varrho}>0$ initially. The gauge-fixed theta expressions extend
holomorphically across the fixed divisor; multiplying out the common
denominator in~\eqref{eq:psi-from-tildepsi} proves this directly.

For real $\lambda$, $L,M$ are anti-Hermitian and traceless,
and satisfy
\[
\overline L\begin{pmatrix}0&1\\-1&0\end{pmatrix}
=\begin{pmatrix}0&1\\-1&0\end{pmatrix}L,\qquad
\overline M\begin{pmatrix}0&1\\-1&0\end{pmatrix}
=\begin{pmatrix}0&1\\-1&0\end{pmatrix}M.
\]
Consequently, the row $(-\overline{\psi_2},\overline{\psi_1})$ solves
the same Lax system. Anti-Hermiticity also gives $\rev{\varrho}_s=\rev{\varrho}_t=0$. Thus $\rev{\varrho}$ remains positive
on the connected domain and the displayed matrix is an $\SU$ frame.
Put $J=\operatorname{diag}(\ii/2,-\ii/2)$ and use the residue $C$
from Proposition~\ref{prop:tildepsi-Lax}. The Sym formula gives
\[
\gamma_s=\Lambda_0\Psi J\Psi^{-1},\qquad
\gamma_t=-\Lambda_0^{-1}\Psi C\Psi^{-1}.
\]
Using $\langle X,Y\rangle=-2\operatorname{tr}(XY)$ and
$R^2+|q_t|^2=1$ proves~\eqref{eq:immersion-metric}.
The equation~\eqref{eq:LR} follows as in
Remark~\ref{rem:reconstruction-point}. Differentiating the normalized
frame gives the component formulas; its normalization contributes
$-\rev{\varrho}_\lambda/(2\rev{\varrho})$ to the diagonal and zero off the diagonal.
\end{proof}

For spectral differentiation, the theta series may be differentiated
termwise:
\[
\partial_{u_j}\theta(u)=\sum_{n\in\Z^g}n_j
 \exp(\rev{\tfrac12 n^{\mathsf T}\tau n+n^{\mathsf T}u}).
\]
Indeed $-\Re\tau$ is positive definite. On any compact set of $u$,
each differentiated summand is bounded by a constant times
$(1+|n|)^m e^{-c|n|^2/4}$ for some $c>0$, a summable majorant.
Thus all derivatives converge locally uniformly. In our normalization
there is no additional factor $2\pi\ii$.
Here $\omega_j/d\lambda$ denotes the coefficient of $d\lambda$ in
the local expression for $\omega_j$, equivalently
$\partial_\lambda\Ab_{-,j}$.
Only $\Ab_-$ and the $\Omega_j$ depend on $\lambda$; where the
components are nonzero, the chain rule gives
\begin{align*}
\frac{\partial_\lambda\psi_1}{\psi_1}
 &=\frac\ii2(s\Omega_1'+t\Omega_2')+\Omega_3'
 +\nabla\log\theta(\Ab_--\mathbf W-\mathbf D-\mathbf r)
                      \cdot(\omega_j/d\lambda)_{j=1}^g,\\
\frac{\partial_\lambda\psi_2}{\psi_2}
 &=\frac\ii2(s\Omega_1'+t\Omega_2')
 +\nabla\log\theta(\Ab_--\mathbf W-\mathbf D)
                      \cdot(\omega_j/d\lambda)_{j=1}^g.
\end{align*}
At a component zero one differentiates the products in~\eqref{eq:psi}
directly and uses the nonsingular matrix-entry formulas above.
\endgroup
\section{\rev{Periodicity} conditions}\label{sec:closure}

\rev{In this section we derive spectral conditions for the reconstructed surface to be periodic in the parameters \(s\) and \(t\).}

Let \(\Omega_1\) and \(\Omega_2\) be the Abelian integrals
and \(\Psi(s,t;\lambda)\in\SU\) be the frame constructed from the Baker--Akhiezer function in Section~3, and define \rev{the Lund--Regge surface} by
\begin{equation}\label{eq:Sym}
  \gamma(s,t)
  = \left.
      \rev{\Lambda_0}\bigl(\partial_{\lambda}\Psi(s,t;\lambda)\bigr)\,
      \Psi(s,t;\lambda)^{-1}
    \right|_{\lambda=\Lambda_0},
  \in\mathfrak{su}(2),
\end{equation}
where \(\Lambda_0 > 0\) is a fixed spectral parameter.
\begingroup

Here $s$ has constant speed $\Lambda_0$, so an $s$-period $L$
has arclength $\Lambda_0L$. Theorems~\ref{thm:closure-s}
and~\ref{thm:closure-t} concern a nonempty open immersion domain
$\mathcal U\subset\{q_t\ne0\}$, invariant under the indicated parameter
shift. Use the common real
determinations of $\Omega_1,\Omega_2$ near $P_*$ fixed above.
Such a domain may have several connected components.
Periodicity on it describes complete \rev{periodic} coordinate curves only
when the corresponding full parameter lines belong to the domain.
Corollary~\ref{cor:coordinate-closure} treats a complete regular
coordinate curve without requiring surface immersion along the full line.
\endgroup
\subsection{\rev{Periodicity} in the \rev{parameter \(s\)}}\label{subsec:closure-s}

\begin{theorem}[\rev{Periodicity} in \(s\)]\label{thm:closure-s}
Let \(L>0\) and suppose \(L\mathbf{U} = 4\pi(n_1, \dots, n_g)^{\rev{\mathsf T}}\) where \(n_{i} \in \Z\).
\rev{Suppose $\mathcal U+(L,0)=\mathcal U$. Then
$\gamma(s+L,t)=\gamma(s,t)$ for every $(s,t)\in\mathcal U$ if and only if} \textup{(i)} \(\dd\Omega_1(P) = 0\) and~\textup{(ii)} \(L\,\Omega_1(P)\in 2\pi\Z\) where \(\lambda(P) = \Lambda_0\).
\end{theorem}

\begin{proof}
Fix \(t\) and consider the shift operator in the \rev{parameter \(s\)}
\[
  L^* : \psi(s,t;\lambda) \longmapsto \psi(s+L,t;\lambda)
\]
acting on the Baker--Akhiezer function
\(\psi = (\psi_1,\psi_2)\).
Using the explicit finite-gap representation of \(\psi\) from Section~3 one
checks that \(L^*\) acts diagonally on \((\psi_1,\psi_2)\) as
\begin{equation}\label{eq:L-shift-psi}
  L^*
  (\psi_1(s,t;\lambda),\psi_2(s,t;\lambda))
  =(\psi_1(s,t;\lambda),\psi_2(s,t;\lambda))X(\lambda),
\end{equation}
with
\begin{equation}\label{eq:X-lambda}
  X(\lambda) :=
  \begin{pmatrix}
    \exp\bigl(\tfrac{\ii}{2}L(\Omega_1(\lambda)+\tfrac{E}{2})\bigr) & 0\\[0.3em]
    0 &
    \rev{\exp\bigl(\tfrac{\ii}{2}L(\Omega_1(\lambda)-\tfrac{E}{2})\bigr)}
  \end{pmatrix},
\end{equation}
where \(E\) is the constant appearing in the asymptotics of \(\Omega_1\) at
\(P_{\infty}^{\pm}\).

From \eqref{eq:L-shift-psi} it follows that \(L^*\) acts on \(\Psi\) by
\begin{equation}\label{eq:L-shift-Psi}
  L^*\Psi(s,t;\lambda)
  = M(\lambda)\,\Psi(s,t;\lambda)\,N,
\end{equation}
\rev{where $M(\lambda)$ and $N$ are the shift matrices}
\begin{equation}\label{eq:M-and-N}
  M(\lambda)
  =
  \begin{pmatrix}
    \exp\bigl(\tfrac{\ii}{2}L\Omega_1(\lambda)\bigr) & 0\\[0.3em]
    0 &
    \exp\bigl(-\tfrac{\ii}{2}L\Omega_1(\lambda)\bigr)
  \end{pmatrix},
  \
  N = \begin{pmatrix}
      \exp\bigl(\frac{i}{4}LE\bigr) & 0 \\
      0 & \exp\bigl(-\frac{i}{4}LE\bigr)
    \end{pmatrix}.
\end{equation}

\begingroup

Applying \(L^*\) to the scaled Sym formula gives
\begin{equation}\label{eq:closure-M}
\gamma(s+L,t)=M\gamma(s,t)M^{-1}
+\left.\Lambda_0(\partial_\lambda M)M^{-1}\right|_{\lambda=\Lambda_0}.
\end{equation}
The factor \(N\) disappears because it is independent of \(\lambda\).
If \(M=\pm I\) and \(\partial_\lambda M=0\) at the reconstruction point,
this proves periodicity.

Conversely, suppose the positions agree on $\mathcal U$.
Differentiating~\eqref{eq:closure-M} in $s$ and $t$ shows that
$\operatorname{Ad}_M$ fixes both $\gamma_s$ and $\gamma_t$.
These are linearly independent on the immersion domain. A rotation
fixing two independent vectors is the identity, so $M=\pm I$.
Equation~\eqref{eq:closure-M} then makes the translation zero.
Finally,
\[
M|_{P_*}=\pm I\ \Longleftrightarrow\ L\Omega_1(P_*)\in2\pi\Z,
\qquad
(\partial_\lambda M)M^{-1}
=\frac{\ii L}{2}\operatorname{diag}(\Omega_1',-\Omega_1').
\]
Since \(\lambda\) is a local coordinate at \(P_*\), the derivative vanishes
if and only if \(d\Omega_1(P_*)=0\).
\endgroup
\end{proof}

\subsection{Periodicity in the \rev{parameter \(t\)}}\label{subsec:closure-t}

The \rev{\(t\)-periodicity} of the \rev{surface} is characterized in exactly the same
way, with \rev{\(\Omega_1\) replaced by \(\Omega_2\).}

\begin{theorem}[Periodicity in \(t\)]\label{thm:closure-t}
Let \(T>0\), and let \(\Omega_2\) be the Abelian integral \rev{introduced in Section~\ref{subsec:BA-theta}}.
Suppose \(T\mathbf{V} = 4\pi(m_1, \dots, m_g)^{\rev{\mathsf T}}\) where \(m_{i} \in \Z\).
\rev{Suppose $\mathcal U+(0,T)=\mathcal U$. The surface is periodic
with period $T$ in $t$,}
\[
  \gamma(s,t+T)=\gamma(s,t)\quad\rev{\text{for all }(s,t)\in\mathcal U},
\]
if and only if \textup{(i)} \(\dd\Omega_2(P) = 0\) and~\textup{(ii)} \(T\,\Omega_2(P)\in 2\pi\Z\), \rev{where \(\lambda(P)=\Lambda_0\)}.
\end{theorem}

\begin{proof}
\begingroup

The common theta period gives
\[
\Psi(s,t+T;\lambda)=K(\lambda)\Psi(s,t;\lambda)N_t,
\]
where
\[
K=\operatorname{diag}(e^{\ii T\Omega_2/2},e^{-\ii T\Omega_2/2}),\qquad
N_t=\operatorname{diag}(e^{-\ii TH/4},e^{\ii TH/4}).
\]
Therefore
\(\gamma(s,t+T)=K\gamma(s,t)K^{-1}+\Lambda_0K'K^{-1}\).
If \(K=\pm I\) and \(K'=0\), this gives periodicity. Conversely,
periodicity on $\mathcal U$ makes $\operatorname{Ad}_K$ fix
the two independent tangent vectors. The argument in
Theorem~\ref{thm:closure-s} gives $K=\pm I$, $K'=0$.
These are exactly \(T\Omega_2(P_*)\in2\pi\Z\) and \(d\Omega_2(P_*)=0\),
where \(P_*\) lies over \(\Lambda_0\).
\endgroup
\end{proof}

\begingroup

\rev{\begin{corollary}[\rev{Periodicity} of a complete regular coordinate curve]
\label{cor:coordinate-closure}
Suppose the BA frame is smooth along a full parameter line and the
corresponding coordinate curve is regular there. Under
$L\mathbf U\in4\pi\Z^g$, $L>0$, the $s$-curve at fixed $t$ is
$L$-periodic if and only if $d\Omega_1(P_*)=0$ and
$L\Omega_1(P_*)\in2\pi\Z$. Under $T\mathbf V\in4\pi\Z^g$,
$T>0$, the analogous statement holds for the $t$-curve, with
$d\Omega_2(P_*)=0$ and $T\Omega_2(P_*)\in2\pi\Z$.
Surface immersion along the full line is not required.
\end{corollary}
\begin{proof}
In either shift formula, the additive term is parallel to the axis of
the displayed diagonal rotation. Projecting a periodic curve onto this
axis forces that translation to be zero. A nonidentity rotation would
then fix the entire curve and place it on the axis. A regular curve
on a line cannot be periodic: its continuous signed derivative is
nonzero and has constant sign. Thus the rotation is the identity.
The shift formulas now give precisely the critical-point and phase
conditions above. Conversely, those conditions make the rotation and
translation the identity and zero, respectively.
\end{proof}}
\endgroup


\section{\rev{Explicit formulas for the genus-one case}}\label{sec:genus1}
Assume \(g=1\) with branch points \(\lambda_1,\bar\lambda_1,\lambda_2,\bar\lambda_2\), so the Riemann surface is given by 
\begin{equation*}
  \mu^2 = \prod_{j=1}^2 (\lambda - \lambda_j)(\lambda - \bar\lambda_j).
\end{equation*}

\begin{figure}[ht]
\begin{tikzpicture}[scale=0.75, transform shape]
  
  
  
  \node at (2,1.8) {$\lambda_{2}$}; 
  \node at (2,-1.8) {$\bar\lambda_{2}$};
  \draw (2,1.5) -- (2,-1.5);
  \draw[-<-] (5,0) ellipse (0.8cm and 2.4cm);
  \node at (5,1.8) {$\lambda_{1}$}; 
  \node at (5,-1.8) {$\bar\lambda_{1}$};
  \draw (5,1.5) -- (5,-1.5);
  \node at (6.2,0.3) {$a$};
  
  \draw[-<-]
    (2,1) to[out=180,in=-90] (1,1.8)
    to[out=90,in=180] (3.5,2.5)
    to[out=0,in=90] (6,1.8)
    to[out=-90,in=0](5,1);
  \draw[dashed,-<-]
    (5,1) to[out=180,in=0](2,1);
    \node at (3.5,2.8) {$b$};



  \end{tikzpicture}
  \caption{\rev{A homology} basis for \rev{the genus-one} Riemann surface $\mathcal{R}$}\label{fig:homologybasisg1}
\end{figure}

\subsection{Holomorphic differential and periods}\label{subsec:g1-periods}
\rev{We use the upper-half-plane representatives fixed in
Section~\ref{subsec:spectral-curve} and the transported canonical basis
of Figure~\ref{fig:homologybasisg1}. Thus the moduli below satisfy
$0<p'<1$ and $0<p<1$. Indeed,
$|\lambda_1-\bar\lambda_2|^2-|\lambda_1-\lambda_2|^2
=4\Im\lambda_1\Im\lambda_2>0$.}

Using the standard elliptic parametrization (cf.~\cite{CaliniIvey}, Section~3), one computes
\begin{gather*}
  \int_{a}\frac{\dd\lambda}{\mu}
    = \frac{4\ii}{|\lambda_1 - \bar\lambda_2|}\,K(p),\qquad
  \int_{b}\frac{\dd\lambda}{\mu}
    = -\,\frac{4}{|\lambda_1 - \bar\lambda_2|}\,K(p'),\\
  p^2 = 1 - (p')^2,\qquad
  p' = |h|,\quad
  h = \frac{\lambda_1 - \lambda_2}{\lambda_1 - \bar{\lambda}_2},
\end{gather*}
\rev{Here $K(p)$ and $E(p)$ denote the complete elliptic integrals of
the first and second kinds, respectively, with modulus $p$.
The function $E(p)$ is distinct from the asymptotic constant $E$.}
Thus the normalized holomorphic differential \(\omega\) satisfying \(\int_a\omega = 2\pi\ii\) is given by
\begin{equation*}
  \omega
    = 2\pi\ii\,\left(\frac{|\lambda_1 - \bar\lambda_2|}{4\ii\,K(p)}\right)\frac{\dd\lambda}{\mu}.
\end{equation*}
In particular, the period matrix reduces to the single scalar
\begin{equation*}
  \tau = \int_b \omega = -\,2\pi\,\frac{K(p')}{K(p)}.
\end{equation*}

\rev{We next introduce normalized Abelian differentials of the second and third kinds,
which determine the quantities \(U,V,r\) in the finite-gap formula.}
\begin{align*}
    d\Omega_1 &= \dfrac{\lambda^{2} - \frac{c}{2}\lambda - c_1}{\mu}d\lambda \sim \mp z^{-2}dz \quad P \to P_{\infty}^{\pm}\\
    d\Omega_2 &= -\mu_0\dfrac{\lambda^{-2} - \frac{d}{2}\lambda^{-1} - c_2}{\mu}d\lambda \sim \mp \lambda^{-2}d\lambda \quad P \to P_0^{\pm}\\
    d\Omega_3 &= \dfrac{\lambda - c_3}{\mu}d\lambda \sim \rev{\mp} z^{-1}dz \quad P \to P_{\infty}^{\pm} 
\end{align*}
Here, in the \((\lambda,\mu)\)-coordinates,
\(P_{0}^+=(0,\mu_0)\) and \(P_{0}^-=(0,-\mu_0)\), where \(\mu_0 = |\lambda_1\lambda_2|\), and
\begin{align*}
    c &= \lambda_1 +\bar\lambda_1 + \lambda_2 +\bar\lambda_2,\quad d = \lambda_1^{-1} +\bar\lambda_1^{-1} + \lambda_2^{-1} +\bar\lambda_2^{-1},\\
    c_1 &= \dfrac{1}{2}\left(|\lambda_1-\bar\lambda_2|^2\frac{E(p)}{K(p)}-|\lambda_1|^2-|\lambda_2|^2\right),\\
    c_2 &= \dfrac{1}{2}\left(|\lambda_1^{-1}-\bar\lambda_2^{-1}|^2\frac{E(p)}{K(p)}-|\lambda_1^{-1}|^2-|\lambda_2^{-1}|^2\right),\\
    c_3 &= \rev{\frac{\displaystyle\int_a\lambda\,d\lambda/\mu}
{\displaystyle\int_a d\lambda/\mu}}.
\end{align*}
\rev{These quantities are computed as follows:}
\begin{align*}
    U &= \int_{b}d\Omega_1 = \frac{\pi |\lambda_1 - \bar\lambda_2|}{K(p)},\quad
    V = \int_{b}d\Omega_2 = -\frac{\pi |\lambda_1 - \bar\lambda_2|}{\mu_0K\rev{(p)}} ,\\
    r &= \rev{-\int_b d\Omega_3=\int_{P_\infty^-}^{P_\infty^+}\omega}.  
\end{align*}


\begingroup

\subsection{Divisor and numerical evaluation}
For genus one the unique theta zero modulo $\Gamma$ is
$z_* =\pi\ii+\tau/2$: the terms indexed by $n$ and $-n-1$ cancel
at this value. Choose \(D\in\ii\R\), and define the single point
\(\delta=P_\delta\) by
\[
\Ab_-(P_\delta)=D+z_*\pmod\Gamma.
\]
Since $K_-\equiv-z_*\pmod\Gamma$, this point satisfies
$\Ab_-(\delta)+K_-\equiv D\pmod\Gamma$. The initially selected purely
imaginary $D$ is the representative used throughout.
A degree-one divisor on an elliptic curve is
nonspecial. The remaining checks are \(P_\delta\ne P_\infty^\pm,P_0^\pm\)
and the theta nonvanishing conditions in Section~\ref{subsec:BA-theta}.
Thus (D3) is implemented explicitly by \(\bar D=-D\). The examples use
\(D=0\) and \(D=\pi\ii/2\), respectively.

For $\tau<0$ the Jacobi product gives
\[
\theta(\ii y)=\prod_{n=1}^{\infty}(1-e^{n\tau})
 \left|1+e^{(n-1/2)\tau+\ii y}\right|^2>0\quad(y\in\R).
\]
Since $W+D$ is purely imaginary, the parameter denominator never
vanishes. This does not imply $q_t\ne0$ on the whole parameter plane.

For numerical computation one may avoid all inverse-elliptic branch choices:
normalize the coordinate differentials by their \(a\)-periods and integrate
them and \(\omega\) on the common paths of Section~\ref{subsec:BA-theta}.

For \(\tau<0\), compute the derivatives of \(\theta\) by termwise
differentiation of its convergent series:
\[
\theta^{(k)}(z)=\sum_{n\in\Z}n^k e^{\tau n^2/2+nz},\qquad k=0,1,2,\ldots.
\]
Thus $(\log\theta)'=\theta'/\theta$. With $A=\Ab_-(P)$ and $z$
independent of $\lambda$, the chain rule gives
\[
\partial_\lambda\log\theta(A-z)
=\frac{\theta'(A-z)}{\theta(A-z)}\frac{\omega}{d\lambda},
\qquad
\frac{\omega}{d\lambda}
=\frac{\pi|\lambda_1-\bar\lambda_2|}{2K(p)\mu},
\]
where the logarithmic derivative is used only when $\theta(A-z)\ne0$.
Use these derivatives in Theorem~\ref{thm:representation}, and convert the
matrix entries to Cartesian coordinates by \eqref{eq:R3-su2}.
The theta truncation, quadrature order, and numerical convergence
check used for the figures are specified in Section~\ref{sec:examples}.
\endgroup
\subsection{\rev{Periodicity} conditions in the genus-one case}

In genus one we rewrite the spatial \rev{periodicity} conditions of
Theorem~\ref{thm:closure-s} in an explicit \rev{elliptic-integral} form.

\begin{proposition}[Genus-one $s$-\rev{periodicity} at the reconstruction point $\lambda=\LamR$]
\label{prop:g1-periodicity}
Assume $g=1$ with branch points $\lambda_1,\bar\lambda_1,\lambda_2,\bar\lambda_2$,
and let $d\Omega_1$ be the normalized second-kind differential introduced above.
\rev{Assume the immersion domain and the $s$-shift invariance
of Theorem~\ref{thm:closure-s}. Periodicity is understood on that domain.}
Let
\[
  U := \int_b d\Omega_1 = \frac{\pi|\lambda_1-\bar\lambda_2|}{K(p)},
\]
where $K(p)$ is the complete elliptic integral of the first kind and $p$ is the modulus
defined in \rev{Section~\ref{subsec:g1-periods}}.
Fix $L>0$ such that
\[
  L\,U \in 4\pi\Z,
\]
for instance
\[
  L=L_n:=\frac{4\pi n}{U}
    =\frac{4n\,K(p)}{|\lambda_1-\bar\lambda_2|},
  \qquad n\in\Z_{>0}.
\]

\smallskip
\noindent
Fix a reconstruction point $\LamR>0$ \rev{and let $P_*\in\mathcal R$ lie over $\LamR$, with the common
path convention fixed in Section~\ref{subsec:BA-theta}.}
Then the reconstructed \rev{map} $\gamma$ defined by the Sym formula~\eqref{eq:Sym}
is \rev{periodic} in $s$ with period $L$ if and only if the two \rev{periodicity} conditions
in Theorem~\ref{thm:closure-s} hold at $\rev{P_*}$, namely
\begin{equation}\label{eq:g1-closure-conds}
  d\Omega_1(\rev{P_*})=0,
  \qquad
  L\,\Omega_1(\rev{P_*})\in 2\pi\Z
  \ \ \Bigl(\Longleftrightarrow\ 
  \exp\bigl(\tfrac{\ii L}{2}\Omega_1(\rev{P_*})\bigr)=\pm1\Bigr).
\end{equation}

\smallskip
\noindent
Moreover, in genus one we have
\[
  d\Omega_1
  = \frac{\lambda^{2} - \frac{c}{2}\lambda - c_1}{\mu}\,d\lambda,
\]
so the condition $d\Omega_1(\rev{P_*})=0$ is equivalent to
\begin{equation}\label{eq:g1-criticalpoint}
  \LamR^2-\frac{c}{2}\LamR-c_1=0.
\end{equation}
\smallskip
\noindent
\begingroup

For \(L=L_n\), evaluate the phase without an unspecified inverse-Jacobi
branch by the equivalent condition
\begin{equation}\label{eq:g1-phase-Jacobi}
\exp\left(\frac{\ii L_n}{2}
\int_Q^{\rev{P_*}}\frac{\lambda^2-(c/2)\lambda-c_1}{\mu}\,d\lambda\right)=\pm1.
\end{equation}
The path is the common lift fixed in Section~\ref{subsec:BA-theta}.
If $F$ is an elliptic antiderivative on this lift, use
$\Omega_1(P_*)=F(P_*)-F(Q)$ in the \rev{periodicity} condition;
an arbitrary additive constant generally changes the phase.
\endgroup
\end{proposition}

\begin{proposition}[Genus-one $t$-\rev{periodicity} at the reconstruction point $\lambda=\LamR$]
\label{prop:g1-t-periodicity}
\rev{Use the genus-one spectral data and notation of Proposition~\ref{prop:g1-periodicity}.}
In particular, $\LamR>0$ is the reconstruction point, $\rev{P_*}\in\mathcal R$ satisfies
$\lambda(\rev{P_*})=\LamR$, and $p$ is the elliptic modulus.

Let
\[
  V := \int_b d\Omega_2
   = -\frac{\pi|\lambda_1-\bar\lambda_2|}{\mu_0 K(p)}.
\]
\rev{Assume the immersion domain and the $t$-shift invariance
of Theorem~\ref{thm:closure-t}.
Fix $T>0$} such that $T\,V\in 4\pi\Z$. For $m\in\Z_{>0}$ we may take
\[
  T=T_m:= -\frac{4\pi m}{V}
  =\frac{4m\,\mu_0\,K(p)}{|\lambda_1-\bar\lambda_2|}.
\]

Then the reconstructed \rev{map} $\gamma$ defined by the Sym formula~\eqref{eq:Sym}
is periodic in $t$ with period $T$ if and only if the \rev{periodicity} conditions of
Theorem~\ref{thm:closure-t} hold at $\rev{P_*}$, namely
\begin{equation}\label{eq:g1-t-closure-conds}
  d\Omega_2(\rev{P_*})=0,
  \qquad
  T\,\Omega_2(\rev{P_*})\in 2\pi\Z
  \ \ \Bigl(\Longleftrightarrow\ 
  \exp\bigl(\tfrac{\ii T}{2}\Omega_2(\rev{P_*})\bigr)=\pm1\Bigr).
\end{equation}

Moreover, in genus one
\[
  d\Omega_2
  = -\mu_0\,\frac{\lambda^{-2} - \frac{d}{2}\lambda^{-1} - c_2}{\mu}\,d\lambda,
\]
so $d\Omega_2(\rev{P_*})=0$ is equivalent to
\begin{equation}\label{eq:g1-t-criticalpoint}
  \LamR^{-2}-\frac{d}{2}\LamR^{-1}-c_2=0.
\end{equation}

\smallskip
\noindent
\begingroup

For \(T=T_m\), the corresponding phase is
\begin{equation}\label{eq:g1-t-phase-Jacobi}
\exp\left(-\frac{\ii T_m\mu_0}{2}
\int_Q^{\rev{P_*}}\frac{\lambda^{-2}-(d/2)\lambda^{-1}-c_2}{\mu}\,d\lambda\right)=\pm1.
\end{equation}
Here the path must avoid \(P_0^\pm\) and use the same lift as the other
Abelian integrals. This specifies the same normalization as in Section~\ref{subsec:BA-theta}.
\endgroup
\end{proposition}

\begingroup
\section{Examples and figures}\label{sec:examples}
The following figures illustrate numerical \rev{periodicity} using the component
formulas in Theorem~\ref{thm:representation}.
\rev{The displayed data give numerical approximations to the \rev{periodicity}
conditions, with the residuals specified below. The plotted surfaces
are considered on $q_t\ne0$. The phase equations are solved numerically.
}

In both examples \(Q=\bar\lambda_2\). We evaluate all four integrals
directly along the common lifted path fixed in
Section~\ref{subsec:BA-theta}. Starting at \(Q\), descend vertically
below both cuts, continue horizontally to the right of both cuts,
and then ascend to a positive real \(R\) to their right. The last
segment runs along the positive real arc from \(R\) to the evaluation
point, continuing \(\mu>0\) as specified above. The initial segment is
parametrized quadratically to remove the square-root singularity at \(Q\).
On the first three segments the cut-plane branch is evaluated as
\[
\mu(\lambda)=\prod_{j=1}^2(\lambda-\bar\lambda_j)
\sqrt{\frac{\lambda-\lambda_j}{\lambda-\bar\lambda_j}},
\]
with the principal square root in each factor. The final real arc uses
the continued positive square root, including when its projection
crosses a cut. The Abel map is computed from the same path as
\(\Ab_-(P)=r/2+\int_Q^P\omega\).
The computed integrals are checked against
$\Im r=\pi$ and $\Im\Omega_3=\pi/2$ on the positive real arc.
The constants \(E,H,\beta,r\) use the continuation of this same path
to \(P_\infty^+\). We subtract \(\lambda\) and \(\log\lambda\)
from \(\Omega_1\) and \(\Omega_3\), respectively, and integrate the
regularized remainders after \(x=1/\lambda\), using algebraically
rationalized expressions to avoid cancellation.

\rev{The figures and residual checks use the same displayed branch points,
reconstruction points, and divisor parameters, held fixed throughout.
The figures use IEEE double-precision arithmetic, 256-point
Gauss--Legendre quadrature on each transformed path segment, and the
theta sum $-12\le n\le12$. Repeating the spectral calculation with
512 points changes the position at $(s,t)=(0.137,0.083)$ by less than
$3\times10^{-11}$. This comparison concerns the figure computation.

The residual checks use the displayed decimal branch points and
reconstruction points as inputs to a separate 60-decimal-digit
calculation. The periods are evaluated as $12\pi/U$ and $-12\pi/V$;
the decimals displayed for them are approximations. Path integrals and
regularized tails are computed with arbitrary-precision adaptive
tanh--sinh quadrature. For a derivative of order $k$, the theta sum
is truncated symmetrically at $N$ chosen so that
\[
 \frac{\tau}{2}(N+1)^2+|\Re z|(N+1)+k\log(N+2)<-75\log10.
\]
This is the numerical cutoff used for omitted summands, rather than
a certified bound for the full tail. Theta derivatives and the Abel-map
chain rule give the Lax residuals. The geometric residual is computed
separately by arbitrary-precision numerical differentiation of the
matrix-entry reconstruction, using extra working precision in the
differences. We use the Euclidean norm for endpoint distances and the
Frobenius norm for row and matrix residuals. For the geometric residual, we apply the Frobenius norm to its
$\mathfrak{su}(2)$ matrix under the identification~\eqref{eq:R3-su2}.
All local residuals below are evaluated at $(s,t)=(0.137,0.083)$;
they do not bound the error uniformly over a plotted surface.
}

\subsection{\texorpdfstring{\rev{Numerical $s$-\rev{periodicity}}}{Numerical s-periodicity}}\label{subsec:example-s}
\begin{example}[\rev{Numerical $s$-\rev{periodicity}}]\label{ex:s-periodic}
Take
\[
\lambda_1=0.454+0.324245540366250\,\ii,\qquad
\lambda_2=-0.454+0.095\,\ii,
\]
and their complex conjugates. Choose $D=0$; the theta-zero formula
in Section~\ref{sec:genus1} gives
$\Ab_-(P_\delta)=\pi\ii+\tau/2\pmod\Gamma$.
Here $D$ is the Abelian divisor parameter, not the divisor itself.
Set
\[
\Lambda_0=0.453442324479044,\qquad
L=12\pi/U\approx19.4718287392057.
\]
\rev{With $\lambda_1=0.454+\ii y$ and $\lambda_2=-0.454+0.095\ii$,
let $P_s(y)$ lie on the positive real arc with $\mu>0$, over the positive root of
$\lambda^2-(c/2)\lambda-c_1$. The displayed $y$ approximates a numerical root
of the scalar phase equation
\[
 F_s(y):=\frac{12\pi}{U(y)}\frac{\Omega_1(P_s(y))}{2\pi}-4=0.
\]
The displayed $\Lambda_0$ approximates the positive critical-point root
at the displayed value of $y$.
Corollary~\ref{cor:coordinate-closure} supplies the exact \rev{periodicity}
criterion; for the displayed data, the phase on that arc,}
\(L\Omega_1(P_*)/(2\pi)\), differs from \(4\) by less than
\(6\times10^{-12}\). The computed endpoint distance is below
\(10^{-12}\). The \(s\)-curve is plotted at \(t=0\);
the surface uses \(0\leq t\leq0.1\).
\end{example}
\begin{proof}[Verification of the divisor conditions]

For these branch points, write
$\mu^2=\lambda^4+p_2\lambda^2+p_1\lambda+p_0$.
Here $p_1=2a(b_1^2-b_2^2)\ne0$ for
$\lambda_1=a+\ii b_1$, $\lambda_2=-a+\ii b_2$.
The function $f=\mu-\lambda^2-p_2/2$ has its only pole, of order two,
at $P_\infty^-$. Since $\dim L(2P_\infty^-)=2$ by Riemann--Roch,
$1,f$ form a basis. Each fiber of the degree-two map $f$ is linearly
equivalent to $2P_\infty^-$; hence $P$ is a ramification point of $f$
if and only if $2\Ab_-(P)\in\Gamma$.
These are ramification points of $f$, not of the projection $\lambda$.
At $P_\infty^+$,
\[
\mu=z^{-2}+p_2/2+(p_1/2)z+O(z^2),\qquad
f=(p_1/2)z+O(z^2),
\]
and at $P_0^\pm$ one has $df/d\lambda=p_1/(2\mu(P_0^\pm))\ne0$.
Thus none of these three points is a ramification point of $f$.
For $D=0$, however, $\Ab_-(P_\delta)=z_*$ is nonzero and
$2z_*\in\Gamma$. Therefore $P_\delta$ avoids these three points and
also $P_\infty^-$, whose Abel image is zero.
Evaluating the theta section at $P_\infty^-$ and $P_\infty^+$ now gives
$\theta(D)\theta(D-r)\ne0$.
\end{proof}
\begin{figure}[ht]
\centering
\begin{subfigure}{.45\textwidth}\centering
\includegraphics[width=\linewidth]{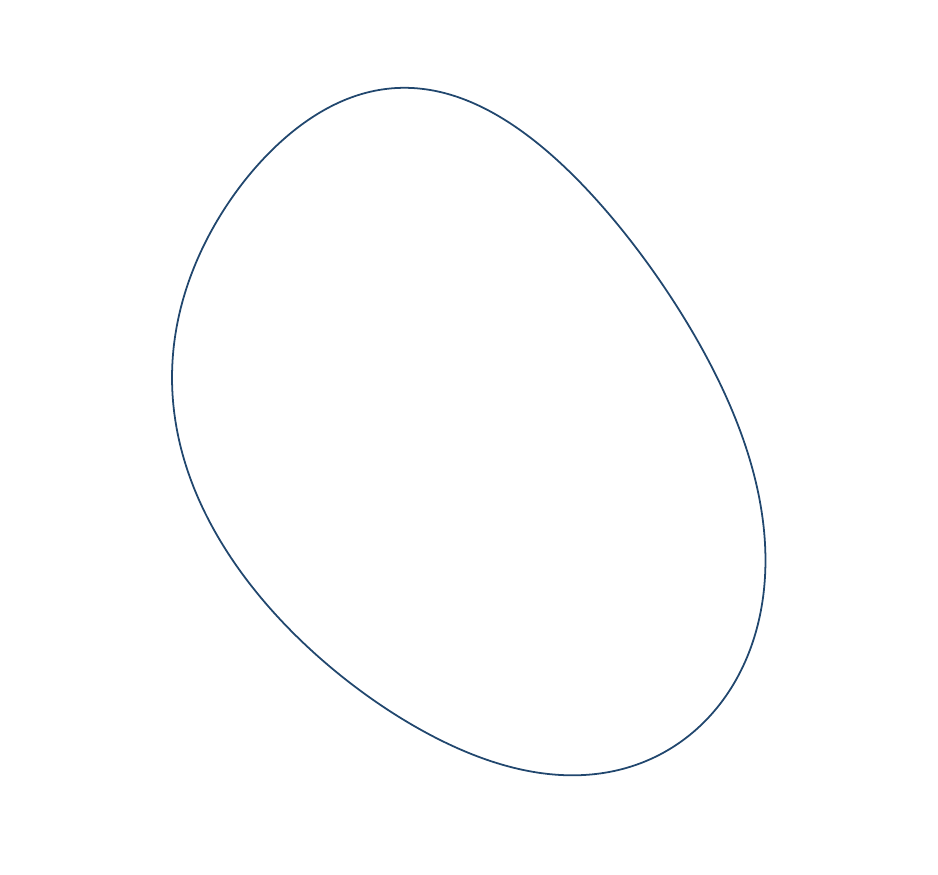}
\caption{\rev{Coordinate curve}}\end{subfigure}\hfill
\begin{subfigure}{.45\textwidth}\centering
\includegraphics[width=\linewidth]{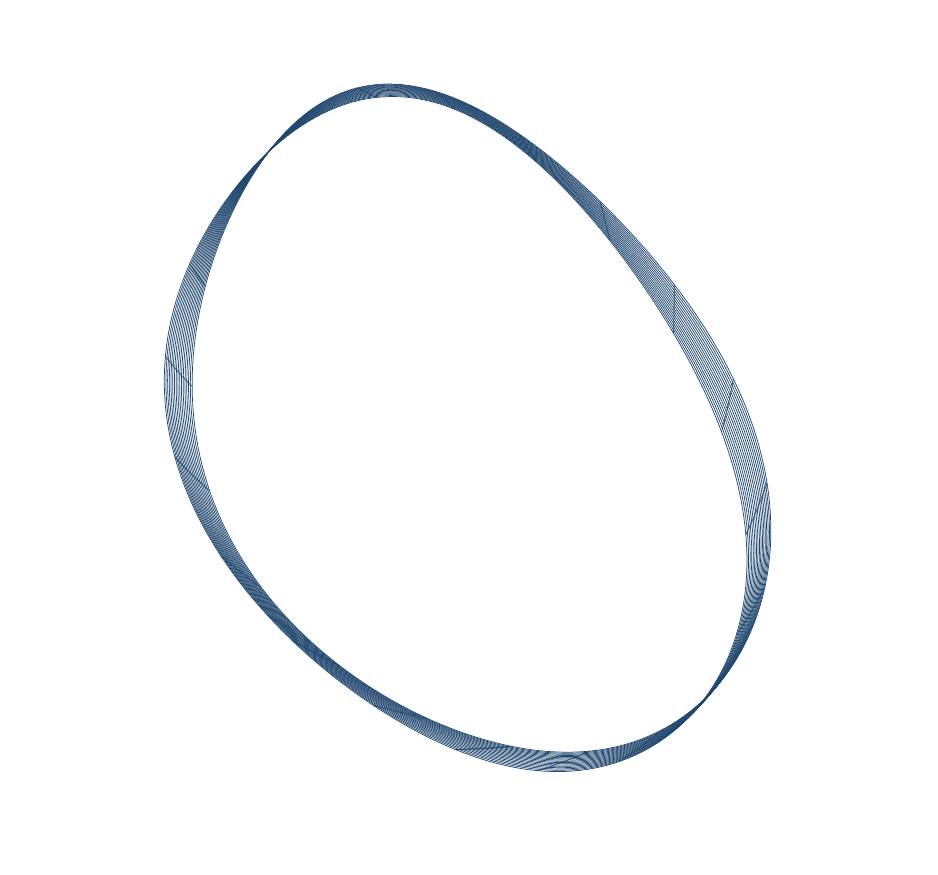}
\caption{\rev{Surface patches}}\end{subfigure}
\caption{\rev{Numerical genus-one $s$-\rev{periodicity}: (a) the coordinate curve at $t=0$,
$0\le s\le L$; (b) the reconstructed surface over
$0\le s\le L$, $0\le t\le0.1$. Both use the scaled Sym formula
at the stated $\Lambda_0$.}}
\label{fig:genus1-closed-3D-s}\label{fig:genus1-surface-3D-s}
\end{figure}

\subsection{\texorpdfstring{\rev{Numerical $t$-\rev{periodicity}}}{Numerical t-periodicity}}\label{subsec:example-t}
\begin{example}[\rev{Numerical $t$-\rev{periodicity}}]\label{ex:t-periodic}
Take
\[
\lambda_1=1+0.813210819353571\,\ii,\qquad
\lambda_2=-1+0.813210819353571\,\ii,
\]
and their complex conjugates. Here \(D=\pi\ii/2\) and
\(\Ab_-(P_\delta)=3\pi\ii/2+\tau/2\pmod\Gamma\).
This choice avoids all four marked points and ensures
$\theta(D)\theta(D-r)\ne0$, as verified below.
For these symmetric data $r=\pi\ii+\tau/2$, as shown below, so
$D=0$ would instead give $P_\delta=P_\infty^+$ and $\theta(D-r)=0$.
Set
\[
\Lambda_0=1.69717959078123,\qquad
T=-12\pi/V\approx13.7361156956655.
\]
\rev{With $\lambda_1=1+\ii y$ and $\lambda_2=-1+\ii y$,
let $P_t(y)$ lie on the positive real arc with $\mu>0$, over the positive root of
$\lambda^{-2}-(d/2)\lambda^{-1}-c_2$. The displayed $y$ approximates a numerical root of
\[
 F_t(y):=-\frac{12\pi}{V(y)}\frac{\Omega_2(P_t(y))}{2\pi}+2=0.
\]
The displayed $\Lambda_0$ approximates the positive critical-point root
at the displayed value of $y$. The exact \rev{periodicity}
criterion is Corollary~\ref{cor:coordinate-closure}; the numerical phase
on that arc,} \(T\Omega_2(P_*)/(2\pi)\), differs from \(-2\) by less than
\(2\times10^{-12}\). The computed endpoint distance is below
\(10^{-11}\). \rev{The plotted curve illustrates numerical \rev{periodicity} of the \(t\)-curve at}
\(s=0\), with \(0\leq t\leq T\).
For these symmetric data, the complete real zero set is
\begin{equation}\label{eq:example-t-singular-lines}
q_t(s,t)=0\quad\Longleftrightarrow\quad
Us+Vt=-(4m+1)\pi,\qquad m\in\mathbb Z.
\end{equation}
Thus the surface is an immersion on the open complement of these lines,
since its metric determinant is \(|q_t|^2\). \rev{The complete $t$-curve is covered by
Corollary~\ref{cor:coordinate-closure}, since the BA denominator is
nonzero on the real parameter plane and both coordinate curves remain
regular even on the excluded lines:} \(|\gamma_s|=\Lambda_0\) and
\(|\gamma_t|=\Lambda_0^{-1}\). The right panel of
Figure~\ref{fig:genus1-surface-3D-t} uses $0\leq s\leq0.1$,
$0\leq t\leq T$, and omits narrow bands around the lines $q_t=0$
to separate the immersed patches. Immersion itself requires only
removal of the lines.
Zeros of \(q\) are distinct from these lines. At a zero of \(q\), the
smooth BA frame and the residue coefficient \(R\) of
Proposition~\ref{prop:tildepsi-Lax} still define the surface, but no
Frenet torsion is assigned there.
\end{example}
\begin{proof}[Verification of the divisor conditions]

For these branch points, the polynomial is even. The involution
$(\lambda,\mu)\mapsto(-\lambda,\mu)$ fixes precisely the four marked
points. It fixes the Abel base point $P_\infty^-$ and pulls $\omega$
back to $-\omega$, so it negates the Abel map modulo $\Gamma$.
Since the genus-one Abel map identifies the curve with its Jacobian,
the four fixed points correspond exactly to the four classes whose
doubles belong to $\Gamma$. For $D=\pi\ii/2$,
$2\Ab_-(P_\delta)\equiv\pi\ii\not\equiv0\pmod\Gamma$,
so this divisor also avoids the marked points. In contrast, $D=0$
would put it at a marked point for these symmetric data.
Hence $\theta(D)\theta(D-r)\ne0$.
\end{proof}
\begin{proof}[Verification of the immersion domain]
For these branch points, the map
\((\lambda,\mu)\mapsto(-\bar\lambda,\bar\mu)\) exchanges the two lower
branch points and fixes \(P_\infty^+\). It pulls each of
\(\eta=\omega,d\Omega_2\) back to \(-\bar\eta\).
Joining the fixed path \(\gamma_\infty\) to the reverse of its image
gives the lower connector between the branch points. The \(b\)-cycle
is this connector minus its sheet exchange, up to a loop at infinity,
whose integral is zero for both differentials. Consequently
\[
\int_b\eta=4\Re\int_{\gamma_\infty}\eta,
\qquad r=\tau/2+\pi\ii,\qquad H=V/2,
\]
where \(\Im r=\pi\) and \(H\in\mathbb R\) were established above.
Put
\[
u=\frac{K(p)}2-\frac{K(p)}{2\pi}(Us+Vt).
\]
Here $\operatorname{sn},\operatorname{cn},\operatorname{dn}$ denote the
Jacobi elliptic functions with modulus $p$. The theta half-period
identity and the Jacobi theta quotients give
\[
q(s,t)=2\sqrt{\beta pp'}\,
 e^{-\tau/8+\pi\ii/4-\ii(E+U/2)s/2}
 \frac{\operatorname{sn}(u,p)}{\operatorname{dn}(u,p)}.
\]
Its prefactor is nonzero and independent of \(t\). Since
\[
\frac{d}{du}\frac{\operatorname{sn}(u,p)}{\operatorname{dn}(u,p)}
=\frac{\operatorname{cn}(u,p)}{\operatorname{dn}(u,p)^2},
\qquad u_t=-\frac{VK(p)}{2\pi}\ne0,
\]
and \(\operatorname{dn}(u,p)\geq p'>0\) for real \(u\), the zeros are
exactly \(u=(2m+1)K(p)\). This proves
\eqref{eq:example-t-singular-lines}. At these values
\(\operatorname{sn}(u,p)=(-1)^m\), so \(q\ne0\).
\end{proof}
\begin{figure}[ht]
\centering
\begin{subfigure}{.45\textwidth}\centering
\includegraphics[width=\linewidth]{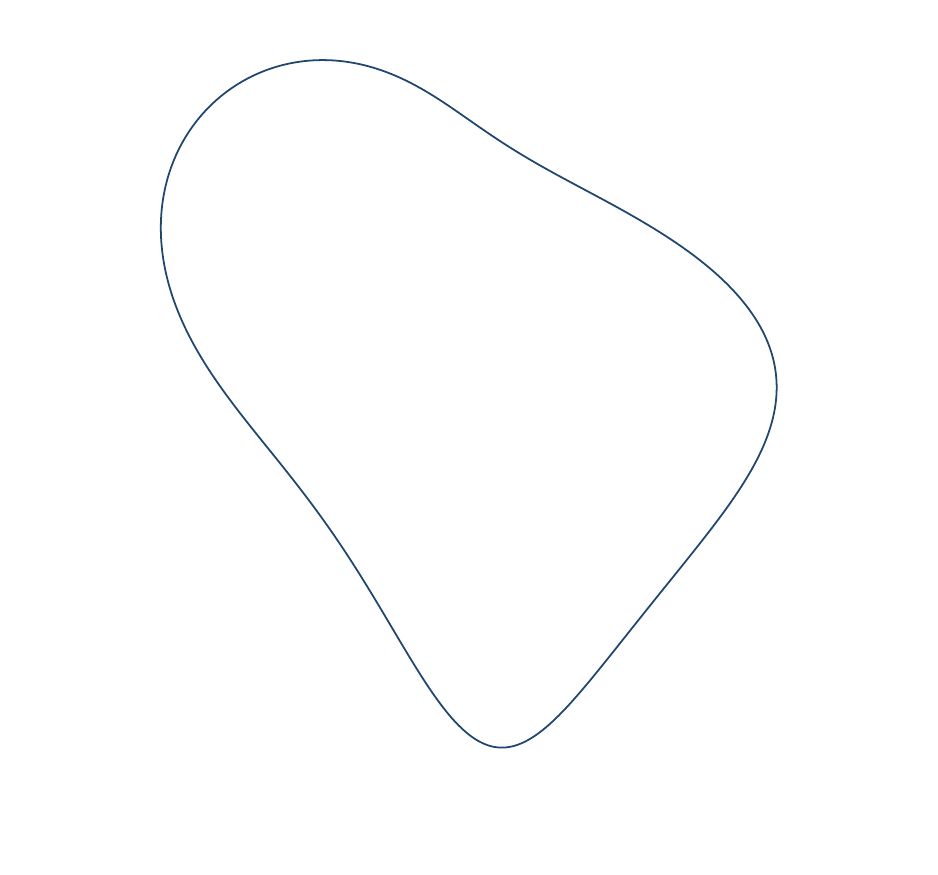}
\caption{\rev{Coordinate curve}}\end{subfigure}\hfill
\begin{subfigure}{.45\textwidth}\centering
\includegraphics[width=\linewidth]{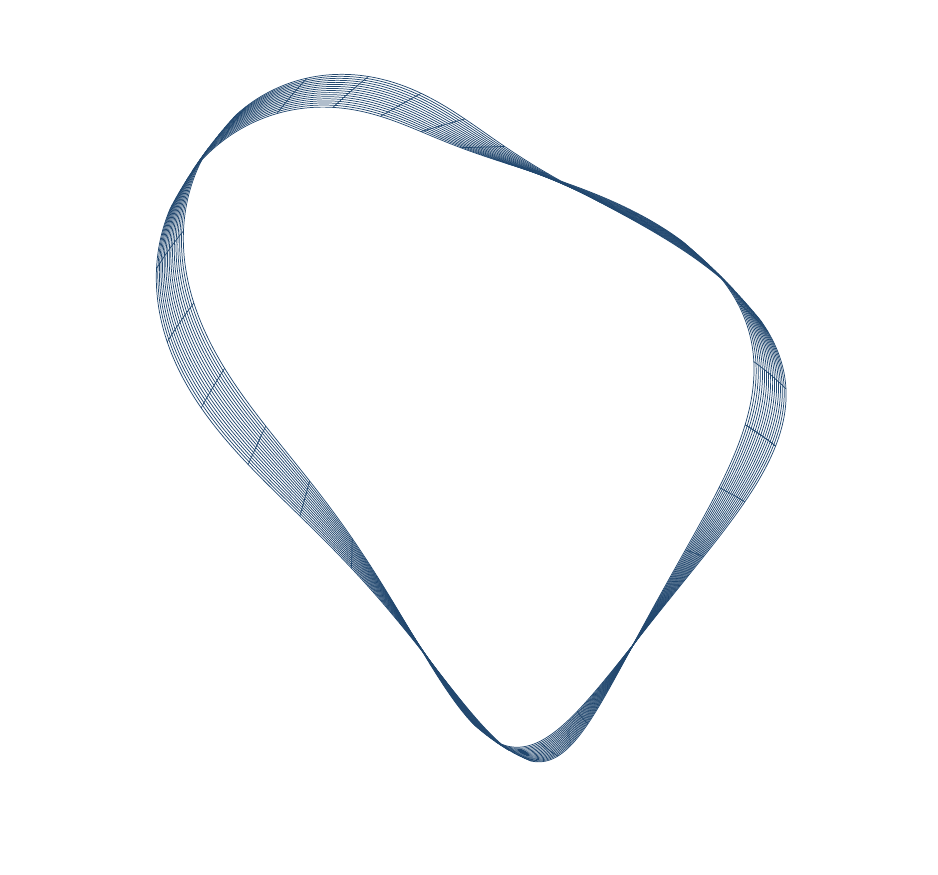}
\caption{\rev{Surface patches}}\end{subfigure}
\caption{\rev{Numerical genus-one $t$-\rev{periodicity}: (a) the coordinate curve at $s=0$,
$0\le t\le T$; (b) immersed patches over $0\le s\le0.1$,
$0\le t\le T$. The surface panel excludes bands
$|(t+Us/V)/T-(4m+1)/12|<0.01$ around the lines $q_t=0$.
The complete coordinate curve crosses those lines regularly.}}
\label{fig:genus1-closed-3D-t}\label{fig:genus1-surface-3D-t}
\end{figure}

The divisor conditions are verified after each example, and the
nonvanishing of $\theta(W+D)$ is proved in Section~\ref{sec:genus1}.
At \((s,t)=(0.137,0.083)\), differentiated theta series for the Lax
equations and high-precision numerical differentiation for
\(\gamma_{st}-\gamma_s\times\gamma_t\) give residuals below
\(10^{-57}\) in both examples. These numerical checks supplement
the analytic construction and \rev{periodicity} criteria.

\endgroup
\FloatBarrier
\begingroup
\appendix
\section{Reflected cycles on the punctured spectral curve}
\label{app:reflected-cycles}
The compact-surface identities suffice for holomorphic differentials.
For the third-kind differential $d\Omega_3$ of Section~\ref{subsec:BA-theta}, however,
the nonzero residues at $P_\infty^\pm$ give nonzero integrals around
small puncture loops. Such loops vanish in compact homology, so that
homology alone does not determine these integrals. We therefore
retain the punctures when computing the reflected cycles.
Let $a_{g+1}$ encircle the omitted cut clockwise on $S_-$,
in the same convention as $a_1,\ldots,a_g$, and let $\ell_+$ be a large counterclockwise
circle on $S_+$ in the $\lambda$-plane, hence a clockwise loop about
$P_\infty^+$ in $z=1/\lambda$. Then
\[
a_{g+1}+\sum_{k=1}^g a_k\equiv\ell_+
\quad\text{in }H_1(\mathcal R\setminus\{P_\infty^+,P_\infty^-\},\Z).
\]
Figure~\ref{fig:cycle-reflection}(a) shows these loops;
$a_{g+1}$ is an auxiliary cycle, not an additional basis element.
In that panel the solid $a$-loops are clockwise on $S_-$,
whereas the dashed outer loop $\ell_+$ is counterclockwise on $S_+$.
Replacing each $a$-loop by its homologous counterclockwise
representative on $S_+$ gives the displayed relation.
Reflection reverses the orientation of $a_j$, and sheet exchange
changes its homology class to its negative. Consequently
\[
\rho_*a_j\equiv-a_j,\qquad \iota_*a_j\equiv-a_j,
\qquad \sigma_*a_j\equiv a_j.
\]
To compute $\rho_*b_j-b_j$, subtract each original arc from its
reflection, as shown separately on the two sheets in
Figure~\ref{fig:cycle-reflection}(b,c). Both arcs in (b) lie on
$S_+$ and are dashed; both arcs in (c) lie on $S_-$ and are solid.
Labels distinguish the original arc from its $\rho$-image.
The dotted horizontal line is the real axis. Bank endpoints are
slightly separated in the drawing, and opposite banks are glued
between the sheets. The cut labels are indices; one intervening cut
is shown schematically, and it is omitted when there is none.
The arrows show the orientations of $b_j$ and $\rho_*b_j$ themselves;
to form their difference, the original arc is traversed backwards.
On $S_+$ this difference encircles the intervening cuts
clockwise. On $S_-$ it encircles those cuts and both endpoint cuts
clockwise. A cut loop on $S_-$ represents the negative of the equally
oriented loop on $S_+$. Since the drawn $a$-cycles are clockwise
on $S_-$, the contributions in this basis are
\[
\begin{array}{c|ccc}
 &\text{omitted cut}&\text{each intervening cut }k&\text{cut }j\\ \hline
 S_+&0&-a_k&0\\
 S_-&a_{g+1}&+a_k&a_j\\ \hline
 \text{sum}&a_{g+1}&0&a_j
\end{array}
\]
The short bank connectors also cancel after gluing. Hence
\[
\rho_*b_j\equiv b_j+a_{g+1}+a_j
\equiv b_j+\ell_+-\sum_{k\ne j}a_k.
\]
For these cut representatives, $b_j+\iota_*b_j$ bounds the
preimage of the bounded planar region between the projected arcs.
Thus $\iota_*b_j\equiv-b_j$, also on the punctured surface.
Likewise, the two equally oriented lifts of a cut loop, or of the
outer loop, bound the preimage of the corresponding bounded region;
hence $\iota_*a_j\equiv-a_j$ and $\iota_*\ell_+\equiv-\ell_+$.
These minus signs concern homology classes: sheet exchange itself
preserves the projected direction of a path. Applying $\iota_*$ yields
\[
\sigma_*b_j\equiv-b_j+\sum_{k\ne j}a_k-\ell_+
\quad\text{in }H_1(\mathcal R\setminus\{P_\infty^+,P_\infty^-\},\Z).
\]
Filling the punctures kills $\ell_+$ and gives the displayed compact
homology relation. For $d\Omega_3$, however, the term $-\ell_+$ must
be retained: its integral is $-2\pi\ii$, since the residue at
$P_\infty^+$ is $-1$ and $\ell_+$ is clockwise in $z$.
When integrating $d\Omega_2$, paths and their deformations also avoid
$P_0^\pm$; additional loops about these points have zero integral
because the residues vanish.

\paragraph{Relative paths and the Abel integral.}
First use the vertical-cut configuration of
Section~\ref{subsec:spectral-curve}.
The path $\gamma_R$, followed by $-\rho\gamma_R$ and then $-\xi$,
is the counterclockwise outer boundary of a bounded slit region on
$S_+$. Its left boundary is the right bank of the omitted cut;
the cuts $1,\ldots,g$ lie inside the region, whereas the omitted
cut and both infinities lie outside it. Remove small neighborhoods
of these $g$ cuts. Each resulting inner boundary is clockwise on
$S_+$ and hence represents $-a_k$, since $a_k$ is counterclockwise
on $S_+$. The oriented boundary of the remaining region therefore gives
\[
 \gamma_R-\rho\gamma_R-\xi-\sum_{k=1}^g a_k\equiv0
 \quad\text{in }H_1(\mathcal R\setminus\{P_\infty^+,P_\infty^-\},\Z).
\]
Here the three open paths together form a closed chain; equivalently,
the required identity holds in relative homology with endpoints
$\{Q,\bar Q\}$. Small rounded connectors at the branch points
shrink to zero. This construction takes place in a bounded region,
so it introduces no $\ell_+$ term. If the region contains a point
over zero, remove a small disk before integrating $d\Omega_2$;
the extra boundary integral is zero because its residue is zero.
The transport of cuts and paths specified in Section~\ref{subsec:spectral-curve}
preserves this boundary calculation.

Next extend $\gamma_R$ to $P_\infty^+$ as in
Section~\ref{subsec:BA-theta}. Choose its lower part below both the
lower $b$-arcs and the small $a$-loops. Then
$-\iota\gamma_\infty+\gamma_\infty$, oriented from
$P_\infty^-$ through $Q$ to $P_\infty^+$, misses all basis cycles:
on $S_+$ it passes below the lower arcs, and on $S_-$ it passes below
the upper return arcs and the $a$-loops. In making the canonical
dissection, join the nested cycle representatives by thin connecting
arcs away from this path. The connecting arcs are traversed in opposite
directions on paired sides and do not alter the cycle periods.
The resulting polygon contains the whole path and both infinity
points in its interior. Thus its single-valued Abel primitive satisfies
\[
 \Ab_{-,j}(P_\infty^-)=0,\qquad
 \Ab_{-,j}(P_\infty^+)
 =\int_{-\iota\gamma_\infty+\gamma_\infty}\omega_j=r_j.
\]

For completeness, pair the polygon sides in the boundary integral of
$\Ab_{-,j}\,d\Omega_3$ and then apply the residue theorem. With
$a_k\cdot b_l=\delta_{kl}$ this gives
\begin{align*}
 \sum_{k=1}^g\left[
  \left(\int_{a_k}\omega_j\right)\left(\int_{b_k}d\Omega_3\right)
 -\left(\int_{b_k}\omega_j\right)\left(\int_{a_k}d\Omega_3\right)
 \right]
 &=2\pi\ii\sum_{P=P_\infty^\pm}
   \operatorname{Res}_P(\Ab_{-,j}\,d\Omega_3)\\
 &=-2\pi\ii r_j.
\end{align*}
The left side is $2\pi\ii\int_{b_j}d\Omega_3$, because the
$a$-periods of $d\Omega_3$ vanish. Hence
$r_j=-\int_{b_j}d\Omega_3$ for the fixed lifts, without an additional
period. This is the bilinear-relation argument of
\cite[Section~5.3]{CaliniIvey}, with the connecting path made explicit.

\endgroup
\begin{figure}[htbp]
\centering
\begingroup
\begin{tikzpicture}[x=1cm,y=.72cm,>=Stealth,
 cyclearrow/.style={postaction={decorate},decoration={markings,
  mark=at position .25 with {\arrow{Stealth}},
  mark=at position .75 with {\arrow{Stealth}}}}]
 \node at (3.15,3.75) {(a) Cut loops on $S_-$ and the outer loop on $S_+$};
 \foreach \x/\lab in {0/{g+1},2.4/{1},6.3/{g}} {
  \draw[line width=1.3pt] (\x,-.8)--(\x,.8);
  \draw[cyclearrow] (\x+.53,0)
   arc[start angle=0,end angle=-360,x radius=.53,y radius=1.23];
  \node at (\x,-1.56) {$a_{\lab}$};
 }
 \node at (4.1,0) {$\cdots$};
 \draw[thick,dashed,cyclearrow] (7.95,0)
  arc[start angle=0,end angle=360,x radius=4.8,y radius=2.52];
 \node[above] at (3.15,2.52) {$\ell_+$};
 \node at (3.15,-3.65)
  {$a_{g+1}+\displaystyle\sum_{k=1}^g a_k\equiv\ell_+$};
\end{tikzpicture}

\medskip
\begin{minipage}[t]{.48\textwidth}
\centering
\begin{tikzpicture}[x=1cm,y=.85cm,>=Stealth,
 arcarrow/.style={postaction={decorate},decoration={markings,
  mark=at position .5 with {\arrow{Stealth}}}}]
 \path[use as bounding box] (-.9,-3.7) rectangle (4.9,3.0);
 \node at (2,2.8) {(b) The two arcs on $S_+$};
 \draw[dotted](-.7,0)--(4.7,0);
 \foreach \x/\lab in {0/{g+1},2/{k},4/{j}} {
  \draw[line width=1.3pt](\x,-.85)--(\x,.85);
  \node[above,font=\small] at (\x,.86) {$\lab$};
 }
 \draw[thick,dashed,arcarrow] (.12,0)
  .. controls (.48,0) and (.45,-1.65) .. (.9,-1.65)
  --(3.1,-1.65)
  .. controls (3.55,-1.65) and (3.52,0) .. (3.88,0);
 \draw[thick,dashed,arcarrow] (.12,0)
  .. controls (.48,0) and (.45,1.65) .. (.9,1.65)
  --(3.1,1.65)
  .. controls (3.55,1.65) and (3.52,0) .. (3.88,0);
 \fill (.12,0) circle[radius=1pt];
 \fill (3.88,0) circle[radius=1pt];
 \node[below,font=\small] at (2,-1.72) {arc of $b_j$};
 \node[above,font=\small] at (2,1.72) {$\rho$-image};
 \node at (2,-3.25) {contribution: $-\displaystyle\sum_{\text{intervening }k}a_k$};
\end{tikzpicture}
\end{minipage}\hfill
\begin{minipage}[t]{.48\textwidth}
\centering
\begin{tikzpicture}[x=1cm,y=.85cm,>=Stealth,
 arcarrow/.style={postaction={decorate},decoration={markings,
  mark=at position .5 with {\arrow{Stealth}}}}]
 \path[use as bounding box] (-.9,-3.7) rectangle (4.9,3.0);
 \node at (2,2.8) {(c) The two arcs on $S_-$};
 \draw[dotted](-.7,0)--(4.7,0);
 \foreach \x/\lab in {0/{g+1},2/{k},4/{j}} {
  \draw[line width=1.3pt](\x,-.85)--(\x,.85);
  \node[above,font=\small] at (\x,.86) {$\lab$};
 }
 \draw[thick,arcarrow] (4.12,0)
  .. controls (4.65,0) and (4.7,1.65) .. (4.1,1.65)
  --(-.1,1.65)
  .. controls (-.7,1.65) and (-.65,0) .. (-.12,0);
 \draw[thick,arcarrow] (4.12,0)
  .. controls (4.65,0) and (4.7,-1.65) .. (4.1,-1.65)
  --(-.1,-1.65)
  .. controls (-.7,-1.65) and (-.65,0) .. (-.12,0);
 \fill (4.12,0) circle[radius=1pt];
 \fill (-.12,0) circle[radius=1pt];
 \node[above,font=\small] at (2,1.72) {arc of $b_j$};
 \node[below,font=\small] at (2,-1.72) {$\rho$-image};
 \node at (2,-3.25) {contribution: $a_{g+1}+a_j+\displaystyle\sum_{\text{intervening }k}a_k$};
\end{tikzpicture}
\end{minipage}

\medskip
\caption{Auxiliary cycles and reflection of $b_j$.}
\label{fig:cycle-reflection}
\endgroup
\end{figure}

\FloatBarrier
\section*{Acknowledgments}
The author would like to express sincere gratitude to Professor Shimpei Kobayashi
for valuable guidance, insightful discussions, and continuous encouragement
throughout this work.

\begingroup
\section*{Declaration of generative AI use}
The author used Codex (OpenAI) to assist with English editing, the
development and revision of mathematical arguments, and the preparation
and debugging of numerical code and diagrams. The author reviewed and
revised the material retained in this article and is responsible for its
accuracy and integrity.
\endgroup

\bibliographystyle{plain}

 \end{document}